\documentclass [twoside,reqno,   12pt] {amsart}

\usepackage{hyperref}
\usepackage{amsfonts}
\usepackage{amssymb}
\usepackage{a4}
\usepackage{color}

\newtheorem{thm}{Theorem}[section]

\newtheorem{prop}[thm]{Proposition}

\newtheorem{defn}[thm]{Definition}
\newtheorem{rem}[thm]{Remark}

\theoremstyle{definition}

\numberwithin{equation}{section}

\renewcommand{\Re}{\hbox{Re}\,}
\renewcommand{\Im}{\hbox{Im}\,}

\newcommand{\C}{\mathbb{C}}

\renewcommand{\div}{\operatorname{div}}

\newcommand{\R}{\mathbb{R}}

\newcommand{\supp}{\operatorname{supp}}

\newcommand{\tr}{\operatorname{tr}}

\newcommand{\pM}{{\partial M}}

\def\hat{\widehat}
\def\tilde{\widetilde}
\def \bfo {\begin {eqnarray*} }
\def \efo {\end {eqnarray*} }
\def \ba {\begin {eqnarray*} }
\def \ea {\end {eqnarray*} }
\def \beq {\begin {eqnarray}}
\def \eeq {\end {eqnarray}}
\def \supp {\hbox{supp }}

\def \det {\hbox{det}}

\def \p {\partial}

\def\hat{\widehat}
\def\tilde{\widetilde}
\def \bfo {\begin {eqnarray*} }
\def \efo {\end {eqnarray*} }
\def \ba {\begin {eqnarray*} }
\def \ea {\end {eqnarray*} }
\def \beq {\begin {eqnarray}}
\def \eeq {\end {eqnarray}}
\def \supp {\hbox{supp }}

\def \det {\hbox{det}}

\def \p {\partial}

\begin{document}

\title[Inverse boundary problems for biharmonic operators ]{Inverse boundary problems for biharmonic operators in transversally anisotropic
geometries}

\author[Yan]{Lili Yan}

\address
        {Lili Yan, Department of Mathematics\\
         University of California, Irvine\\ 
         CA 92697-3875, USA }

\email{liliy6@uci.edu}

\maketitle

\begin{abstract} 
We study inverse boundary problems for first order perturbations of the biharmonic operator on  a conformally transversally anisotropic Riemannian manifold of dimension $n \ge 3$. We show that a continuous first order perturbation can be determined uniquely from the knowledge of the set of the Cauchy data on the boundary of the manifold provided that the geodesic $X$-ray transform on the transversal manifold is injective. 
\end{abstract}

\section{Introduction and statement of results}

Let $(M,g)$ be a smooth compact oriented Riemannian manifold of dimension $n\ge 3$ with smooth boundary $\p M$. Let  $-\Delta_g$ be the Laplace--Beltrami operator, and $(-\Delta_g)^2$ be the biharmonic operator on $M$. Let $X\in C(M,TM)$ be a complex vector field and let $q\in C(M,\C)$.  In this paper we shall be concerned with an inverse boundary problem for the first order perturbation of the biharmonic operator, 
\[
L_{X,q}=(-\Delta_g)^2+X+q.
\]
Let us now proceed to introduce some notation and state the main result of the paper. Let $u\in H^3(M^{\text{int}})$ be a solution to 
\begin{equation}
\label{eq_int_1}
L_{X,q}u=0\quad \text{in}\quad M.
\end{equation}
Here and in what follows $H^s(M^{\text{int}})$, $s\in\R$, is the standard Sobolev space on $M^{\text{int}}$, and $M^{\text{int}}=M\setminus\p M$ stands for the interior of $M$. Let $\nu$ be the unit outer normal to $\p M$. We shall define the trace of the normal derivative $\p_\nu(\Delta_g u)\in H^{-1/2}(\p M)$ as follows. Let $\varphi\in H^{1/2}(\p M)$. Then letting $v\in H^1(M^{\text{int}})$ be a continuous extension of $\varphi$, we set
\begin{equation}
\label{eq_int_2}
\langle \p_\nu (-\Delta_g u), \varphi \rangle_{H^{-1/2}(\p M)\times H^{1/2}(\p M)}=\int_M \big( \langle \nabla_g (-\Delta_g u), \nabla_g v\rangle_g +X(u)v+quv\big)dV_g,
\end{equation}
where $dV_g$ is the Riemannian volume element on $M$. As $u$ satisfies \eqref{eq_int_1}, the definition of the trace $\p_\nu(\Delta_g u)$ on $\p M$ is independent of the choice of an extension $v$ of $\varphi$. Associated to \eqref{eq_int_1}, we define the set of the Cauchy data, 
\begin{equation}
\label{eq_int_3}
\mathcal{C}_{X,q}=\{(u|_{\p M}, (\Delta_g u)|_{\p M}, \p_\nu u|_{\p M}, \p_\nu(\Delta_g u)|_{\p M}): u\in H^3(M^{\text{int}}), \, L_{X,q}u=0 \text{ in }M\}.
\end{equation}
Note that the first two elements in the set of the Cauchy data $\mathcal{C}_{X,q}$ correspond to the Navier boundary conditions for the first order perturbation of the biharmonic operator. Physically, such operators  arise when considering the equilibrium configuration of an elastic plate which is hinged along the boundary, see  \cite{Gazzola_2010}. One can also define the set of the Cauchy data  for the first order perturbation of the biharmonic operator, based on the Dirichlet boundary conditions $(u|_{\p M}, \p_\nu u|_{\pM})$, which corresponds to the clamped plate equation, 
\[
\tilde {\mathcal{C}}_{X,q}=\{(u|_{\p M}, \p_\nu u|_{\p M}, \p_\nu^2 u|_{\p M}, \p_\nu^3 u|_{\p M}): u\in H^3(M^{\text{int}}), \, L_{X,q}u=0 \text{ in }M\}.
\]
The explicit description for the Laplacian in the boundary normal coordinates shows that $\mathcal{C}_{X,q}=\tilde {\mathcal{C}}_{X,q}$, see \cite{Lee_Uhlmann}, \cite{Krup_Lassas_Uhlmann_func_anal}. 

The inverse boundary problem that we are interested in is to determine the vector field $X$ and the potential $q$ from the knowledge of the set of the Cauchy data $\mathcal{C}_{X,q}$. 

This problem was studied extensively in the Euclidean setting, see  \cite{Krup_Lassas_Uhlmann_trans},  \cite{Krup_Lassas_Uhlmann_func_anal},   \cite{Assylbekov_2016}, \cite{Assylbekov_Iyer}, \cite{Ikehata_1991} \cite{Isakov_1991} \cite{Bhattacharyya_Ghosh_2019}, \cite{Bhattacharyya_Ghosh_2020}, \cite{Ghosh_2015}, \cite{Ghosh_Krishnan_2015}, \cite{Yang_Yang_2014}. Specifically,    it was shown in  \cite{Krup_Lassas_Uhlmann_trans} that the set of the Cauchy data $\mathcal{C}_{X,q}$ determines the vector field $X$ and the potential $q$ uniquely. 
Let us note that the unique determination of a first order perturbation of the Laplacian is not possible due to the gauge invariance of boundary measurements and in this case the first order perturbation can be recovered only modulo a gauge transformation, see \cite{NakSunUlm_1995}, \cite{Sun_1993}.

Going beyond the Euclidean setting, inverse boundary problems for lower order perturbations of the Laplacian were only studied in the case when $(M,g)$ is CTA  (conformally transversally anisotropic, see Definition \ref{def_CTA} below) and under the assumption that the geodesic $X$-ray transform on the transversal manifold is injective, see  the fundamental works \cite{DKSaloU_2009} and \cite{DKuLS_2016} which initiated this study, and see also \cite{DKSU_2007}, \cite{DKSalo_2013},  \cite{KK_2018},  \cite{KK_advection_2018}, \cite{Cekic}.
\begin{defn}
\label{def_CTA}
A compact Riemannian manifold $(M,g)$ of dimension $n\ge 3$ with boundary $\p M$ is called conformally transversally anisotropic if $M\subset\subset \R\times M_0^{\text{int}}$  where $g= c(e \oplus g_0)$, $(\R,e)$ is the Euclidean real line, $(M_0,g_0)$ is a smooth compact $(n-1)$--dimensional manifold with smooth boundary, called the transversal manifold, and  and $c\in C^\infty(M)$ is a positive function. 
\end{defn}
The injectivity of the geodesic $X$-ray transform is known when the manifold $(M_0,g_0)$ is simple, in the sense that any two points in $M_0$ are connected by a unique geodesic depending smoothly on the endpoints and that $\p M_0$ is strictly convex, see \cite{Anikonov_78}, \cite{Muhometov}, when $M_0$ has strictly convex boundary and is foliated by strictly convex hypersurfaces \cite{Stefanov_Uhlmann_Vasy}, \cite{Uhlmann_Vasy_2016}, and also when $M_0$ has a hyperbolic trapped set and no conjugate points \cite{Guillarmou_2017}, \cite{Guill_Mazz_Tzou}. An example of the latter occurs when $M_0$ is negatively curved manifold.

Turning the attention to the inverse boundary problem of determining the first order perturbation of the biharmonic operator, this problem was solved in  \cite{Assylbekov_Yang_2017} in the case when $(M,g)$ is CTA and the transversal manifold $(M_0,g_0)$ is simple, extending the result of \cite{DKSaloU_2009} to the case of biharmonic operators. To be on par with the best results available for the perturbations of the Laplacian in the context of Riemannian manifolds, the goal of this paper is to solve the inverse problem for the first order perturbation of the biharmonic operator in the case when $(M,g)$ is CTA and the geodesic $X$-ray transform is injective on the transversal manifold $(M_0,g_0)$, generalizing the result of \cite{DKuLS_2016} to the case of biharmonic operators.

Let us recall some definitions related to geodesic $X$-ray transform following \cite{Guillarmou_2017}, \cite{DKSaloU_2009}. The geodesics on $M_0$ can be parametrized by points on the unit sphere bundle $SM_0=\{(x,\xi)\in TM_0: |\xi|=1\}$.  Let 
\[
\p_\pm SM_0=\{ (x,\xi)\in SM_0: x\in \p M_0, \pm \langle\xi,  \nu(x) \rangle>0\},
\] 
be the incoming ($-$) and outgoing ($+$) boundaries of $SM_0$. Here  $\nu$ is the unit outer normal vector field to $\p M_0$. Here and in what follows $\langle \cdot,\cdot\rangle$ is the duality between $T^*M_0$ and $TM_0$.

Let $(x,\xi)\in \p_-SM_0$ and $\gamma=\gamma_{x,\xi}(t)$ be the geodesic on $M_0$ such that $\gamma(0)=x$ and $\dot{\gamma}(0)=\xi$. Let us denote by $\tau(x,\xi)$ the first time when the geodesic  $\gamma$ exits $M_0$ with the convention that $\tau(x,\xi)=+\infty$ if the geodesic does not exit $M_0$. We define the incoming tail by
\[
\Gamma_-=\{(x,\xi)\in \p_-SM_0:\tau(x,\xi)=+\infty\}.
\]
When $f\in C(M_0,\C)$ and $\alpha\in C(M_0,T^*M_0)$ is a complex valued $1$-form, we define the geodesic $X$-ray transform on $(M_0,g_0)$  as follows
\[
I(f,\alpha)(x,\xi)=\int_0^{\tau(x,\xi)} \big[ f(\gamma_{x,\xi}(t))+ \langle \alpha(\gamma_{x,\xi}(t)), \dot{\gamma}_{x,\xi}(t) \rangle \big]dt, \quad (x,\xi)\in \p_-SM_0\setminus\Gamma_-.
\]

A unit speed geodesic segment $\gamma=\gamma_{x,\xi}:[0,\tau(x,\xi)]\to M_0$, $\tau(x,\xi)>0$, is called non-tangential if $\gamma(0),\gamma(\tau(x,\xi))\in \p M_0$, $\dot{\gamma}(0),\dot{\gamma}(\tau(x,\xi))$ are non-tangential vectors on $\p M_0$, and $\gamma(t)\in M_0^{\text{int}}$ for all $0<t<\tau(x,\xi)$.

\textbf{Assumption 1.} \textit{We assume that the geodesic $X$-ray transform on $(M_0,g_0)$ is injective in the sense that  if $I(f,\alpha)(x,\xi)=0$ for all $(x,\xi)\in \p_-SM_0\setminus\Gamma_-$  such that  $\gamma_{x,\xi}$ is a non-tangential geodesic then  $f=0$ and $\alpha=dp$ in $M_0$ for some $p\in C^1(M_0,\C)$ with $p|_{\p M_0}=0$}. 

The main result of the paper is as follows. 
\begin{thm}
\label{thm_main}
Let $(M,g)$ be a CTA manifold of dimension $n\ge 3$ such that Assumption 1 holds for the transversal manifold. Let $X^{(1)}, X^{(2)}\in C(M, TM)$ be complex vector fields, and let $q^{(1)},q^{(2)}\in C(M,\C)$. If $\mathcal{C}_{X^{(1)},q^{(1)}}=\mathcal{C}_{X^{(2)},q^{(2)}}$ then $X^{(1)}=X^{(2)}$ in $M$. Assuming furthermore that 
\begin{equation}
\label{eq_int_boundaty}
q^{(1)}|_{\p M}=q^{(2)}|_{\p M},
\end{equation}
we have $q^{(1)}=q^{(2)}$ in $M$.
\end{thm}
\begin{rem}
Examples of non-simple manifolds $M_0$ satisfying Assumption 1 include in particular manifolds  with  strictly convex boundary which are foliated by strictly convex hypersurfaces \cite{Stefanov_Uhlmann_Vasy}, \cite{Uhlmann_Vasy_2016}, and manifolds with a hyperbolic trapped set and no conjugate points \cite{Guillarmou_2017}, \cite{Guill_Mazz_Tzou}. 

\end{rem}

\begin{rem}
To the best of our knowledge, Theorem \ref{thm_main} seems to be the first result where one recovers a vector field uniquely on general CTA manifolds. 
\end{rem}
\begin{rem}
The assumption \eqref{eq_int_boundaty} is made for simplicity only and can be removed by performing the boundary determination as done in Appendix \ref{sec_boundary_rec} for the vector fields $X^{(1)}$ and $X^{(2)}$. This can be done by using the approach of \cite{Guill_Tzou_duke} combined with its extensions in \cite{KK_advection_2020} and \cite{FKOU_2020}. 
\end{rem}

Let us proceed to describe the main ideas in the proof of Theorem \ref{thm_main}. The key step in the proof is a construction of complex geometric optics solutions for the equations $L_{X,q}u=0$ and $L_{-\overline{X},-\div(\overline{X})+\overline{q}}u=0$ in $M$. Here the operator $L_{-\overline{X},-\div(\overline{X})+\overline{q}}$ represents the formal $L^2$ adjoint of the operator $L_{X,q}$.  In contrast to the work \cite{Assylbekov_Yang_2017}, where one deals with the same inverse problem in the case of a simple transversal manifold,  here without simplicity assumption, complex geometric optics solutions cannot be easily constructed by means of a global WKB method and following  \cite{DKuLS_2016}, we shall construct complex geometric optics solutions based on Gaussian beam quasimodes for the biharmonic operator $(-\Delta_g)^2$ conjugated by an exponential weight corresponding to the limiting Carleman weight $\phi(x)=\pm x_1$ for $-h^2\Delta_g$ on the CTA manifold $(M,g)$, see \cite{DKSaloU_2009}.  To convert the Gaussian beam quasimodes to exact solutions, we shall rely on the corresponding Carleman estimate with a gain of two derivatives established in \cite{KK_2018}, see also \cite{DKSaloU_2009}.  

\begin{rem}
We would like to note that one can obtain Gaussian beam quasimodes for the biharmonic operator $(-\Delta_g)^2$ conjugated by an exponential weight as the Gaussian beams quasimodes for the Laplacian conjugated by an exponential weight. However, such quasimodes are not enough to prove Theorem \ref{thm_main} as in order to recover the vector field uniquely, one has to exploit a richer set of amplitudes which are not available for the Gaussian beams quasimodes for the Laplacian. 
\end{rem}

\begin{rem}
When constructing  Gaussian beams quasimodes for the Laplacian conjugated by an exponential weight, one first reduces to the setting when the conformal factor $c=1$ by using the following transformation, 
\[
c^{\frac{n+2}{4}}\circ (-\Delta_g)\circ c^{-\frac{(n-2)}{4}} = -\Delta_{\tilde{g}} +\tilde{q}, 
\]
where
\[
\tilde g=e\oplus g_0,\quad \tilde q=-c^{\frac{n+2}{4}}(-\Delta_g) (c^{-\frac{(n-2)}{4}}),
\] 
see \cite{DKuLS_2016}. However, it seems that no such useful reduction is available for the biharmonic operator and therefore, when constructing  Gaussian beams quasimodes for the biharmonic operator $(-\Delta_g)^2$ conjugated by an exponential weight, we shall proceed directly accommodating the conformal factor in the construction which makes it somewhat more complicated. 
\end{rem}

Once complex geometric optics solutions are constructed, the next step is to substitute them into a suitable integral identity which is obtained as a consequence of the equality $\mathcal{C}_{X^{(1)},q^{(1)}}=\mathcal{C}_{X^{(2)},q^{(2)}}$ for the Cauchy data sets.  Exploiting the concentration properties of the corresponding Gaussian beams together with Assumption 1, we first show that there exists $\psi\in C^1(\R\times M_0)$ with compact support in $x_1$ such that $\psi(x_1,\cdot)|_{\p M_0}=0$ and $X^{(1)}-X^{(2)}=\nabla_g \psi$. To show that $\psi=0$, i.e. $X^{(1)}=X^{(2)}$, we use the concentration properties of the Gaussian beams for the biharmonic operator with a richer set of amplitudes which are not available for the Laplacian, combining with Assumption 1. Finally, we show that $q^{(1)}=q^{(2)}$ by using the concentration properties of the Gaussian beams together with Assumption 1 once again. 

The plan of the paper is as follows. In Section \ref{sec_Gaussian_beams} we construct Gaussian beam quasimodes for the biharmonic operator conjugated by an exponential weight corresponding to the limiting Carleman weight $\phi$ and establish some concentration properties of them. In Section \ref{sec_CGO} we convert the Gaussian beam quasimodes to the exact complex geometric optics solutions. Section \ref{sec_proof_of_main_result} is devoted to the proof of Theorem \ref{thm_main}. Finally, in Appendix \ref{sec_boundary_rec} the boundary determination of a continuous vector field on a compact manifold with boundary, from the set of the Cauchy data, is presented.

\section{Gaussian beams quasimodes for biharmonic operators on conformally anisotropic manifolds}
\label{sec_Gaussian_beams}

Let $(M,g)$ be  a conformally transversally anisotropic manifold so that $(M,g)\subset\subset (\R\times M_0^{\text{int}}, c(e\oplus g_0))$. Let us write $x=(x_1,x')$ for local coordinates in $\R\times M_0$. Note that $\phi(x)=\pm x_1$ is a limiting Carleman weight for $-h^2\Delta_g$, see \cite{DKSaloU_2009}.  

In this section we shall construct Gaussian beam quasimodes for the biharmonic operator $(-\Delta_g)^2$ conjugated by an exponential weight corresponding to the limiting Carleman weight $\phi$.  Thanks to the presence of the conformal factor $c$, our quasimodes will be constructed on the manifold $M$ and will be localized to non-tangential geodesics on the transversal manifold $M_0$. 

The first main result of this section is as follows. 
\begin{prop}
\label{prop_Gaussian_beams}
Let $s=\mu+i\lambda$ with $1\le \mu=1/h$ and $\lambda\in\R$ being fixed, and let  $\gamma:[0,L]\to M_0$ be a unit speed  non-tangential geodesic on $M_0$. Then there exist  families  of Gaussian beam quasimodes $v_s, w_s\in C^\infty(M)$ such that 
\begin{equation}
\label{eq_prop_gaussian_1}
\|v_s\|_{H^1_{\emph{\text{scl}}}(M^{\emph{\text{int}}})}=\mathcal{O}(1),\quad \| e^{sx_1}(-h^2\Delta_g)^2 e^{-sx_1}v_s\|_{L^2(M)}=\mathcal{O}(h^{5/2}),
\end{equation}
and
\begin{equation}
\label{eq_prop_gaussian_2}
\|w_s\|_{H^1_{\emph{\text{scl}}}(M^{\emph{\text{int}}})}=\mathcal{O}(1),\quad \| e^{-sx_1}(-h^2\Delta_g)^2 e^{sx_1}w_s\|_{L^2(M)}=\mathcal{O}(h^{5/2}),
\end{equation}
as $h\to 0$.  Moreover, in a sufficiently small neighborhood  $U$ of a point $p\in \gamma([0,L])$, the quasimode $v_s$ is a finite sum, 
\[
v_s|_{U}=v_s^{(1)}+\cdots+ v_s^{(P)},
\] 
where $t_1<\cdots <t_P$ are the times in $[0,L]$ where $\gamma(t_l)=p$. Each $v_s^{(l)}$ has the form
\begin{equation}
\label{eq_quasimode_v_0-form}
v^{(l)}_s=e^{is\varphi^{(l)}}a^{(l)}, \quad l=1,\dots, P,
\end{equation}
where $\varphi=\varphi^{(l)}\in C^\infty(\overline{U};\C)$ satisfies for $t$ close to $t_l$, 
\[
\varphi(\gamma(t))=t, \quad \nabla \varphi(\gamma(t))=\dot{\gamma}(t), \quad \emph{\text{Im}}\,(\nabla^2\varphi(\gamma(t)))\ge 0, \quad  \emph{\text{Im}}\,(\nabla^2\varphi)|_{\dot{\gamma}(t)^\perp}>0, 
\] 
and $a^{(l)}\in C^\infty(\R\times \overline{U})$ is of the form, 
\[
a^{(l)}(x_1,t,y)=h^{-\frac{(n-2)}{4}}a_0^{(l)}(x_1,t)\chi\bigg(\frac{y}{\delta'}\bigg),
\]
where for all $l=1,\dots, P$,  either
\begin{equation}
\label{eq_sec_2_ampl_a_1st}
a_0^{(l)}= e^{-\phi^{(l)}(x_1,t)}, 
\end{equation}
is an amplitude of the first type, or $a_0$ satisfies the equation 
\begin{equation}
\label{eq_sec_2_ampl_a_2st}
\frac{1}{c(x_1,t,0)} (\partial_{x_1}-i\partial_t) (e^{\phi^{(l)}(x_1,t)} a^{(l)}_0)=1,
\end{equation}
defining an amplitude of the second type. Here 
\begin{equation}
\label{eq_sec_2_phi}
\phi^{(l)}(x_1,t)=\log c(x_1,t,0)^{\frac{n}{4}-\frac{1}{2}}+G^{(l)}(t), \quad  \partial_t G^{(l)}(t)=\frac{1}{2}(\Delta_{g_0}\varphi^{(l)})(t,0), 
\end{equation}
$(t,y)$ are the Fermi coordinates for $\gamma$ for $t$ close to $t_l$, $\chi\in C^\infty_0(\R^{n-2})$ is such that $0\le \chi\le 1$,   $\chi=1$ for $|y|\le 1/4$ and $\chi=0$ for $|y|\ge 1/2$, and $\delta'>0$ is a fixed number that can be taken arbitrarily small. 

In a sufficiently small neighborhood  $U$ of a point $p\in \gamma([0,L])$, the quasimode $w_s$ is a finite sum, 
\[
w_s|_{U}=w_s^{(1)}+\cdots+ w_s^{(P)},
\] 
where $t_1<\cdots <t_P$ are the times in $[0,L]$ where $\gamma(t_l)=p$. Each $w_s^{(l)}$ has the form
\begin{equation}
\label{eq_quasimode_w_0-form}
w^{(l)}_s=e^{is\varphi^{(l)}}b^{(l)}, \quad l=1,\dots, P,
\end{equation}
where $\varphi^{(l)}$ is the same as in \eqref{eq_quasimode_v_0-form}, and $b^{(l)}\in C^\infty(\R\times \overline{U})$ is of the form
\[
b^{(l)}(x_1,t,y) = h^{-\frac{(n-2)}{4}}b_0^{(l)}(x_1,t)\chi \bigg( \frac{y}{\delta'} \bigg), 
\]
where 
\begin{equation}
\label{eq_sec_2_ampl_b_1st}
b_0^{(l)} = e^{-\tilde \phi^{(l)}(x_1,t)}. 
\end{equation}
Here
\begin{equation}
\label{eq_sec_2_tilde_phi}
\tilde \phi^{(l)}(x_1,t)=\log c(x_1,t,0)^{\frac{n}{4}-\frac{1}{2}}+F^{(l)}(t), \quad \partial_t F^{(l)}(t)=\frac{1}{2}(\Delta_{g_0}\varphi^{(l)})(t,0).
\end{equation}

\end{prop}

\begin{proof}
To construct Gaussian beam quasimodes, we shall follow the standard approach, see \cite{DKuLS_2016}, \cite{KK_2018}. The novelty here is that when working with the biharmonic operator we have to accommodate the presence of the conformal factor $c$ throughout the construction.  We are also led to consider a richer class of amplitudes for the Gaussian beams quasimodes. 

\textit{Step 1. Preparation.}

Let us isometrically embed the manifold $(M_0,g_0)$ into a larger closed manifold $(\hat{M_0},g_0)$ of the same dimension. This is possible as we can form the manifold $\hat{M_0}=M_0\sqcup_{\p M_0} M_0$, which is the disjoint union of two copies of $M_0$, glued along the boundary.  We extend $\gamma$ as a unit speed geodesic in $\hat{M_0}$. Let $\varepsilon>0$ be such that $\gamma(t)\in 
\hat{M_0}\setminus M_0$ and $\gamma(t)$ has no self-intersection for $t\in [-2\varepsilon, 0)\cup (L,L+2\varepsilon]$. This choice of $\varepsilon$ is possible since $\gamma$ is non-tangential. 

Our aim is to construct Gaussian beam quasimodes near $\gamma([-\varepsilon, L+\varepsilon])$. We shall start by carrying out the quasimode construction locally near a given point $p_0=\gamma (t_0)$ on  $\gamma([-\varepsilon, L+\varepsilon])$. Let $(t,y)\in U=\{ (t,y)\in \R\times \R^{n-2}: |t-t_0|<\delta, |y|<\delta'\}$, $\delta, \delta'>0$,  be Fermi coordinates near $p_0$, see \cite{Kenig_Salo_APDE_2013}. We may assume that the coordinates $(t, y)$ extends smoothly to a neighborhood of $\overline{U}$. 
The geodesic $\gamma$ near $p_0$ is then given by $\Gamma=\{(t,y): y=0\}$, and 
\[
g_0^{jk}(t,0)=\delta^{jk},\quad \p_{y_l} g_0^{jk}(t, 0)=0. 
\]
Hence, near the geodesic 
\begin{equation}
\label{eq_gauss_3_metric}
g_0^{jk}(t,y)=\delta^{jk}+\mathcal{O}(|y|^2). 
\end{equation}

Let us first construct the quasimode $v_s$ in \eqref{eq_prop_gaussian_1} for the operator $e^{sx_1} (-h^2 \Delta_g)^2 e^{-sx_1}$. In doing so, we consider the following Gaussian beam ansatz, 
\begin{equation}
\label{eq_gauss_4_v_s}
v_s(x_1, t,y)=e^{i s\varphi(t,y)} a(x_1, t,y; s).
\end{equation}
Here  $\varphi\in C^\infty (U,\C)$ is  such that 
\begin{equation}
\label{eq_gauss_4_v_s_phase}
\Im \varphi\ge 0,\quad \Im \varphi|_{\Gamma}=0,  \quad \Im \varphi(t,y)\sim |y|^2= \text{dist} ((y,t), \Gamma)^2,
\end{equation}
and $a \in C^\infty(\R\times U, \C)$ is an amplitute  such that $\text{supp}(a(x_1,\cdot))$ is close to  $\Gamma$, see \cite{Ralston_1982}, 
\cite{KKL_book}.  
Notice that here we choose  $\varphi$ to depend on the transversal variables $(t,y)$ only while $a$ is a function of all the variables. 

Let us first compute $e^{sx_1} (-h^2 \Delta_g)^2 e^{-sx_1} v_s$.
To that end, letting 
\begin{equation}
\label{eq_200_0}
\tilde \varphi(x_1,t, y)=x_1-i\varphi(t,y), \quad \hat \varphi=sh\tilde \varphi,
\end{equation}
we first get 
\begin{equation}
\label{eq_200_1}
e^{\frac{\hat\varphi}{h}}(-h^2\Delta_g)e^{-\frac{\hat\varphi}{h}}=-h^2\Delta_g +h(2 \langle \nabla_g \hat \varphi, \nabla_g \cdot \rangle_g +\Delta_g \hat \varphi)-\langle \nabla_g \hat \varphi, \nabla_g \hat \varphi\rangle_g.
\end{equation}
Here and in what follows we write $\langle\cdot , \cdot\rangle_g$ to denote the Riemannian scalar product on tangent and cotangent spaces. In view of \eqref{eq_200_1}, we see that 
\[
e^{s\tilde \varphi}(-h^2\Delta_g)^2e^{-s\tilde \varphi}=h^4\big(-\Delta_g +s(2 \langle \nabla_g \tilde \varphi, \nabla_g \cdot \rangle_g +\Delta_g \tilde \varphi)-s^2\langle \nabla_g \tilde \varphi, \nabla_g \tilde \varphi\rangle_g\big)^2,
\]
and therefore, 
\begin{equation}
\label{eq_200_2}
e^{sx_1} (-h^2 \Delta_g)^2 e^{-sx_1} v_s=e^{is\varphi}h^4\big(-\Delta_g +s(2 \langle \nabla_g \tilde \varphi, \nabla_g \cdot \rangle_g +\Delta_g \tilde \varphi)-s^2\langle \nabla_g \tilde \varphi, \nabla_g \tilde \varphi\rangle_g\big)^2a.
\end{equation}

\textit{Step 2. Solving an eikonal equation to determine the phase function $\varphi(t,y)$.}

Following the WKB method, we start by considering the eikonal equation 
\[
\langle \nabla_g \tilde \varphi, \nabla_g \tilde \varphi\rangle_g=0,
\]
and we would like to find $\varphi=\varphi(t,y)\in C^\infty(U,\C)$ such that 
\begin{equation}
\label{eq_200_3}
\langle \nabla_g \tilde \varphi, \nabla_g \tilde \varphi\rangle_g=\mathcal{O}(|y|^3), \quad y\to 0, 
\end{equation}
and 
\begin{equation}
\label{eq_200_3_2}
\text{Im}\varphi\ge d |y|^2,
\end{equation}
with some $d>0$.  Using that $g=c(e\otimes g_0)$ and \eqref{eq_200_0}, we see that 
\[
\langle \nabla_g \tilde \varphi, \nabla_g \tilde \varphi\rangle_g=c^{-1}(1-\langle \nabla_{g_0}  \varphi, \nabla_{g_0}  \varphi\rangle_{g_0}),
\]
and therefore, in view of \eqref{eq_200_3}, we have to find $\varphi$ satisfying the standard eikonal equation, 
\[
1-\langle \nabla_{g_0}  \varphi, \nabla_{g_0}  \varphi\rangle_{g_0}=\mathcal{O}(|y|^3), \quad y\to 0.
\]
As in  \cite{DKuLS_2016}, \cite{Ralston_1977} and \cite{Ralston_1982}, we can choose,  
\begin{equation}
\label{eq_gauss_9}
\varphi(t,y)=t+\frac{1}{2} H(t) y\cdot y,
\end{equation}
where $H(t)$ is a unique smooth complex symmetric solution of the initial value problem for the matrix Riccati equation, 
\begin{equation}
\label{eq_Riccati}
\dot{H}(t)+H(t)^2=F(t), \quad H(t_0)=H_0,
\end{equation}
with $H_0$ being a complex symmetric matrix with $\Im (H_0)$ positive definite and $F(t)$ being a suitable symmetric matrix. Hence, as explained in \cite{DKuLS_2016}, \cite{Ralston_1977} and \cite{Ralston_1982},  $\Im (H(t))$ is  positive definite for all $t$.

\textit{Step 3. Solving a  transport equation to find an amplitude $a$.}
We look for a smooth amplitude $a=a(x_1,x')$ satisfying the transport equation,
\begin{equation}
\label{eq_200_9}
L^2a=\mathcal{O}(|y|), 
\end{equation}
as $y\to 0$.  Here 
\begin{equation}
\label{eq_200_4_L}
L:=2 \langle \nabla_g \tilde \varphi, \nabla_g \cdot \rangle_g +\Delta_g \tilde \varphi.
\end{equation}
To proceed let us first simplify the operator $L$. To that end, in view of \eqref{eq_200_0}, a direct computation shows that 
\begin{equation}
\label{eq_200_5}
\langle \nabla_g \tilde \varphi, \nabla_g \cdot \rangle_g= \frac{1}{c}(\p_{x_1}-ig_0^{-1}(x')\varphi'_{x'}\cdot\p_{x'} ), 
\end{equation}
\begin{equation}
\label{eq_200_6}
\Delta_g \tilde \varphi=\Delta_g x_1-i \Delta_g \varphi(x'), 
\end{equation}
where 
\begin{equation}
\label{eq_200_7}
\Delta_g x_1 = \bigg(\frac{n}{2}-1\bigg) \frac{1}{c^2} \partial_{x_1}c,
\end{equation}
and 
\begin{equation}
\label{eq_200_8}
\Delta_g \varphi = \frac{1}{c} \Delta_{g_0} \varphi + \bigg(\frac{n}{2} -1\bigg) \frac{1}{c^2}  \langle \nabla_{g_0}c, \nabla_{g_0} \varphi \rangle_{g_0}.
\end{equation}
In view of  \eqref{eq_200_5},  \eqref{eq_200_6},  \eqref{eq_200_7},  \eqref{eq_200_8},  the operator $L$ given by \eqref{eq_200_4_L} becomes 
\begin{equation}
\label{eq_200_10}
L=\frac{2}{c}(\p_{x_1}-ig_0^{-1}(x')\varphi'_{x'}\cdot\p_{x'}) +\bigg(\frac{n}{2}-1\bigg) \frac{1}{c^2} \partial_{x_1}c-\frac{i}{c} \Delta_{g_0} \varphi - \bigg(\frac{n}{2} -1\bigg) \frac{i}{c^2}  \langle \nabla_{g_0}c, \nabla_{g_0} \varphi \rangle_{g_0}.
\end{equation}
Let us proceed to simplify the operator $L$ further. Using \eqref{eq_gauss_3_metric} and 
\eqref{eq_gauss_9}, we see that 
\begin{equation}
\label{eq_200_11}
g_0^{-1}(x')\varphi'_{x'}\cdot\p_{x'}=\p_t +\mathcal{O}(|y|^2)\p_t+H(t)y\cdot \p_y+\mathcal{O}(|y|^2)\cdot \p_y.
\end{equation}
Using \eqref{eq_gauss_3_metric} and \eqref{eq_gauss_9}, we also have
\begin{align*}
(\Delta_{g_0}\varphi)(t,0)&=|g_0|^{-1/2}\p_{x'_j}(|g_0|^{1/2} g_0^{jk}\p_{x'_k}\varphi)|_{y=0} =\delta^{jk}\p_{x'_j}\p_{x'_k}\varphi|_{y=0}\\
&=\delta^{jk}H_{jk}=\tr H(t),
\end{align*}
and therefore
\begin{equation}
\label{eq_200_11_2}
(\Delta_{g_0}\varphi)(t,y)= (\Delta_{g_0}\varphi)(t,0)+\mathcal{O}(|y|) = \tr H(t) + \mathcal{O}(|y|).
\end{equation}
Finally, using  \eqref{eq_gauss_3_metric} and \eqref{eq_gauss_9}, we get 
\begin{equation}
\label{eq_guassian_inner}
\begin{aligned}
\langle \nabla_{g_0}c, \nabla_{g_0} \varphi \rangle_{g_0} =  \partial_t c+\mathcal{O}(|y|).
\end{aligned}
\end{equation}

Using \eqref{eq_200_11}, \eqref{eq_200_11_2}, \eqref{eq_guassian_inner}, the operator $L$ in 
\eqref{eq_200_10} becomes
\begin{equation}
\begin{aligned}
\label{eq_gauss_14}
L = &\frac{2}{c}\bigg[  \partial_{x_1} - i \partial_t - iH(t)y\cdot\partial_y+ \bigg(\frac{n}{4}-\frac{1}{2}\bigg)(\partial_{x_1}-i\partial_t)\log c - \frac{i}{2}\tr H(t)\\
& +\mathcal{O}(|y|)+\mathcal{O}(|y|^2)\p_t +\mathcal{O}(|y|^2)\p_y \bigg]\\
&=\frac{2}{c(x_1,t,0)}\bigg[ \partial_{x_1} - i \partial_t  - iH(t)y\cdot\partial_y + (\partial_{x_1}-i\partial_t)\log c(x_1,t,0)^{\frac{n}{4}-\frac{1}{2}}\\
&  -\frac{i}{2}\tr H(t)+\mathcal{O}(|y|)+\mathcal{O}(|y|)(\partial_{x_1},\partial_t) + \mathcal{O}(|y|^2) \partial_y \bigg].
\end{aligned}
\end{equation}

Let $\chi\in C^\infty_0(\R^{n-2})$ be such that  $\chi=1$ for $|y|\le 1/4$ and $\chi=0$ for $|y|\ge 1/2$.  We look for the  amplitude $a$ in the form
\begin{equation}
\label{eq_gauss_10}
a(x_1,t,y)=h^{-\frac{(n-2)}{4}}a_0(x_1,t) \chi\bigg(\frac{y}{\delta'}\bigg),
\end{equation}
where $a_0(\cdot, \cdot) \in C^\infty(\R\times  \{t: |t-t_0|<\delta\} )$ is independent of $y$. In view of \eqref{eq_200_9}, $a_0$ should satisfy  the equation 
\begin{equation}
\label{eq_gauss_11}
L^2a_0 = \mathcal{O}(|y|),
\end{equation}
as $ y\to 0$. 
In view of \eqref{eq_gauss_14}, we write 
\begin{equation}
\label{eq_gauss_14_L_0}
L=\frac{2}{c(x_1,t,0)}(L_0+R),
\end{equation}
where 
\begin{equation}
\label{eq_gauss_17}
L_0 =  (\partial_{x_1} - i \partial_t )+(\partial_{x_1}-i\partial_t)\log c(x_1,t,0)^{\frac{n}{4}-\frac{1}{2}} -\frac{i}{2}\tr H(t),
\end{equation}
and 
\begin{equation}
\label{eq_gauss_17_R}
R=- iH(t)y\cdot\partial_y+\mathcal{O}(|y|)+\mathcal{O}(|y|)(\partial_{x_1},\partial_t) + \mathcal{O}(|y|^2) \partial_y. 
\end{equation}
To solve our inverse problem, we need two types of amplitudes. Let us proceed to construct the first type of amplitudes. In doing so, first note that as $a_0$ is independent of $y$, if $a_0$ solves the following equation 
\begin{equation}
\label{eq_gauss_16}
L_0a_0 = 0,
\end{equation}
then $a_0$ satisfies \eqref{eq_gauss_11}. Let us proceed to find a solution to \eqref{eq_gauss_16}. To that end, 
letting
\begin{equation}
\label{eq_200_12_0}
\phi(x_1,t)=\log c(x_1,t,0)^{\frac{n}{4}-\frac{1}{2}}+G(t), \quad  \partial_t G(t)=\frac{1}{2}\tr H(t), 
\end{equation}
we see that
\begin{equation}
\label{eq_200_12}
L_0 = e^{-\phi(x_1,t)}(\partial_{x_1}-i\partial_t)e^{\phi(x_1,t)}.
\end{equation}
We solve \eqref{eq_gauss_16} by taking 
\begin{equation}
\label{eq_guass_type1_amplitude}
a_0 = e^{-\phi} = c(x_1,t,0)^{\frac{1}{2}-\frac{n}{4}}e^{-G(t)},\quad \partial_t G(t)=\frac{1}{2}\tr H(t).
\end{equation}

Now we proceed to find the second type of amplitudes, which is given by more general solutions to \eqref{eq_gauss_11}. As $a_0$ is independent of $y$,  using \eqref{eq_gauss_14_L_0}, 
\eqref{eq_gauss_17},  \eqref{eq_gauss_17_R}, the equation 
\eqref{eq_gauss_11} becomes
\[
\frac{2}{c(x_1,t,0)}[L_0+R] \bigg(\frac{2}{c(x_1,t,0)}L_0 a_0(x_1,t)+\mathcal{O}(|y|)\bigg)=\mathcal{O}(|y|),
\]
or simply 
\begin{equation}
\label{eq_200_13}
L_0 \bigg(\frac{1}{c(x_1,t,0)}L_0\bigg) a_0(x_1,t)=0.
\end{equation}
Using \eqref{eq_200_12}, we see that \eqref{eq_200_13} becomes
\begin{equation}
\label{eq_200_14}
(\partial_{x_1}-i\partial_t)\bigg(\frac{1}{c(x_1,t,0)} (\partial_{x_1}-i\partial_t) (e^{\phi(x_1,t)} a_0)\bigg)=0.
\end{equation}
To solve \eqref{eq_200_14}, we choose $a_0(x_1,t)$ to be a solution  to 
\begin{equation}
\label{eq_200_15}
\frac{1}{c(x_1,t,0)} (\partial_{x_1}-i\partial_t) (e^{\phi(x_1,t)} a_0)=1.
\end{equation}
Note that the equation \eqref{eq_200_15} can be solved as it is a standard inhomogeneous $\overline{\p}$ equation in the complex plane $z=x_1-it$, 
\begin{equation}
\label{eq_200_15_2}
\overline{\p} (e^{\phi(x_1,t)}a_0) = c/2. 
\end{equation}

Let us remark that  the first type of the amplitudes, i.e.  solutions $a_0$ of the equation  \eqref{eq_gauss_16}, given by \eqref{eq_guass_type1_amplitude}, will be used to recover the potential $q$  as well as the vector field $X$ up to a suitable gauge transformation,  while to recover $X$ uniquely, we shall have to work with the second type of amplitudes, i.e. solutions $a_0$ of the equation \eqref{eq_200_13} given by \eqref{eq_200_15}.

\textit{Step 4. Establishing the estimates \eqref{eq_prop_gaussian_1} locally near the point $p_0$.}

First it follows from \eqref{eq_gauss_4_v_s} and \eqref{eq_gauss_10} that 
\begin{equation}
\label{eq_200_16}
v_s(x_1, t,y)=e^{i s\varphi(t,y)}h^{-\frac{(n-2)}{4}}a_0(x_1,t) \chi\bigg(\frac{y}{\delta'}\bigg).
\end{equation}
Using \eqref{eq_200_3_2}, we have 
\begin{equation}
\label{eq_200_16_1}
|v_s(x_1,t,y)|\le \mathcal{O}(1)  h^{-\frac{(n-2)}{4}}  e^{-\frac{1}{h} d|y|^2} \chi\bigg(\frac{y}{\delta'}\bigg), \quad (x_1,t,y)\in J\times U,
\end{equation}
and therefore, 
\begin{equation}
\label{eq_200_17}
\|v_s\|_{L^2(J\times U)}\le \mathcal{O}(1) \|  h^{-\frac{(n-2)}{4}} e^{-\frac{1}{h} d|y|^2}  \|_{L^2(|y|\le \delta'/2)}=\mathcal{O}(1),\quad h\to 0, 
\end{equation}
where $J \subset \R$ is a large fixed bounded open interval. Similarly, it follows from 
\eqref{eq_200_16} that 
\begin{equation}
\label{eq_200_18}
\|\nabla v_s\|_{L^2(J\times U)} = \mathcal{O}(h^{-1}).
\end{equation}
Let us next estimate $\| e^{sx_1}(-h^2\Delta_g)^2e^{-sx_1}v_s\|_{L^2(J\times U)}$.  To that end, letting 
\begin{equation}
\label{eq_200_18_1}
f=\langle \nabla_g \tilde \varphi, \nabla_g \tilde \varphi\rangle_g=\mathcal{O}(|y|^3), 
\end{equation}
 cf. \eqref{eq_200_3},  we obtain from \eqref{eq_200_2} with the help of \eqref{eq_200_4_L} that 
\begin{equation}
\label{eq_200_19}
\begin{aligned}
e^{sx_1}&(-h^2\Delta_g)^2e^{-sx_1}v_s= e^{is\varphi}h^4 \big( (-\Delta_g)^2a-s\Delta_g(La)+s^2\Delta_g(fa)\\
&+sL(-\Delta_g a)+s^2L^2a -s^3L(fa)
+s^2f(\Delta_ga)-s^3fLa +s^4f^2a
\big).
\end{aligned}
\end{equation}
We shall proceed to bound each term in \eqref{eq_200_19} in $L^2(J\times U)$. First using \eqref{eq_gauss_10} and \eqref{eq_200_3_2}, we get
\begin{equation}
\label{eq_200_20}
\begin{aligned}
\|e^{is\varphi}h^4  (-\Delta_g)^2a\|_{L^2(J\times U)}&=h^4\|e^{is\varphi}h^{-\frac{(n-2)}{4}} (-\Delta_g)^2(a_0\chi)\|_{L^2(J\times U)}\\
&=\mathcal{O}(h^4)\|h^{-\frac{(n-2)}{4}} e^{-\frac{d}{h}|y|^2}\|_{L^2(|y|\le\delta'/2 )}=\mathcal{O}(h^4),
\end{aligned}
\end{equation}
and similarly, 
\begin{equation}
\label{eq_200_21}
\|e^{is\varphi}h^4  s\Delta_g(La)\|_{L^2(J\times U)}=\mathcal{O}(h^3), 
\end{equation}
and 
\begin{equation}
\label{eq_200_22}
\|e^{is\varphi}h^4  sL (\Delta_g a)\|_{L^2(J\times U)}=\mathcal{O}(h^3). 
\end{equation}
Now to bound $e^{is\varphi}h^4s^2\Delta_g(fa)$ in $L^2(J\times U)$ we note that the worst case occurs when $\Delta_g$ falls on $f$, and in this case we have, using \eqref{eq_200_18_1} and \eqref{eq_gauss_10}, 
\[
\|e^{is\varphi}h^4  s^2\Delta_g(f)a\|_{L^2(J\times U)}\le \mathcal{O}(h^2)\|h^{-\frac{(n-2)}{4}} |y|e^{-\frac{d}{h}|y|^2}\|_{L^2(|y|\le\delta'/2 )}=\mathcal{O}(h^{5/2}),
\]
and therefore, 
\begin{equation}
\label{eq_200_23}
\|e^{is\varphi}h^4s^2\Delta_g(fa)\|_{L^2(J\times U)}=\mathcal{O}(h^{5/2}). 
\end{equation}
Here we have used the following bound
\begin{equation}
\label{eq_200_24}
\|h^{-\frac{(n-2)}{4}} |y|^ke^{-\frac{d}{h}|y|^2}\|_{L^2(|y|\le\delta'/2 )}=\mathcal{O}(h^{k/2}), \quad k=1,2,\dots. 
\end{equation}
Similarly, using \eqref{eq_gauss_11} and \eqref{eq_200_24}, we get 
\begin{equation}
\label{eq_200_25}
\|e^{is\varphi}h^4s^2L^2a\|_{L^2(J\times U)}\le \mathcal{O}(h^2)\|h^{-\frac{(n-2)}{4}} |y|e^{-\frac{d}{h}|y|^2}\|_{L^2(|y|\le\delta'/2 )}=\mathcal{O}(h^{5/2}). 
\end{equation}
Using \eqref{eq_200_18_1}, \eqref{eq_200_24}, and the fact that $L(\mathcal{O}(|y|^3))=\mathcal{O}(|y|^3)$, we obtain that 
\begin{equation}
\label{eq_200_26}
\begin{aligned}
&\|e^{is\varphi}h^4s^3 L(fa)\|_{L^2(J\times U)}\le \mathcal{O}(h)\|h^{-\frac{(n-2)}{4}} |y|^3e^{-\frac{d}{h}|y|^2}\|_{L^2(|y|\le\delta'/2 )}=\mathcal{O}(h^{5/2}),\\
&\|e^{is\varphi}h^4s^2 f(\Delta_g a)\|_{L^2(J\times U)}\le \mathcal{O}(h^2)\|h^{-\frac{(n-2)}{4}} |y|^3e^{-\frac{d}{h}|y|^2}\|_{L^2(|y|\le\delta'/2 )}=\mathcal{O}(h^{7/2}),\\
&\|e^{is\varphi}h^4s^3 fL a\|_{L^2(J\times U)}\le \mathcal{O}(h)\|h^{-\frac{(n-2)}{4}} |y|^3e^{-\frac{d}{h}|y|^2}\|_{L^2(|y|\le\delta'/2 )}=\mathcal{O}(h^{5/2}),\\
&\|e^{is\varphi}h^4s^4 f^2 a\|_{L^2(J\times U)}\le \mathcal{O}(1)\|h^{-\frac{(n-2)}{4}} |y|^6e^{-\frac{d}{h}|y|^2}\|_{L^2(|y|\le\delta'/2 )}=\mathcal{O}(h^{3}).
\end{aligned}
\end{equation}
Combining \eqref{eq_200_19}, \eqref{eq_200_20}, \eqref{eq_200_21}, \eqref{eq_200_22}, \eqref{eq_200_23}, \eqref{eq_200_25}, \eqref{eq_200_26}, we get 
\begin{equation}
\label{eq_200_27}
\| e^{sx_1}(-h^2\Delta_g)^2e^{-sx_1}v_s\|_{L^2(J\times U)}=\mathcal{O}(h^{5/2}).
\end{equation}
This completes verification of \eqref{eq_prop_gaussian_1} locally. 

For the later purposes we need estimates for $\|v_s(x_1,\cdot)\|_{L^2(\p M_0)}$.  If $U$ contains a boundary point $x_0=(t_0,0)\in \p M_0$, then $\p_t |_{x_0}$ is transversal to $\p M_0$. Let $\rho$ be a boundary defining function for $M_0$ so that $\p M_0$ is given by the zero set $\rho(t,y)=0$ near $x_0$. Then $\nabla \rho(x_0)$  is normal to $\p M_0$, and hence, $\p_t \rho(x_0)\ne 0$. By the implicit function theorem, there is a smooth function $y\mapsto t(y)$ near $0$ such that $\p M_0$ near $x_0$ is given by $\{ (t(y), y): |y|<r_0\}$ for some $r_0>0$ small,  see also \cite{Kenig_Salo_APDE_2013}. Then using \eqref{eq_200_16_1}, we get
\begin{equation}
\label{eq_v_s_on_boundary}
\begin{aligned}
\|v_s(x_1,\cdot)\|_{L^2(\p M_0\cap U)}^2&=\int_{|y|<r_0} | v_s(x_1,t(y),y)|^2 dS(y)\\
&\le \mathcal{O}(1) \int_{\R^{n-2}} h^{-\frac{(n-2)}{2}}e^{-2\frac{d}{h} |y|^2}dy=\mathcal{O}(1). 
\end{aligned}
\end{equation}

\textit{Step 5. Establishing estimates  \eqref{eq_prop_gaussian_1} globally.}

Now let us construct the quasimode $v_s$ in $M$ by gluing together quasimodes defined along small pieces of the geodesic.   As $\gamma:(-2\varepsilon, L+2\varepsilon)\to \hat M_0$ is a unit speed non-tangential geodesic, an application of  \cite[Lemma 7.2]{Kenig_Salo_APDE_2013}  shows that $\gamma|_{[-\varepsilon,L+\varepsilon]}$ self-intersects only at finitely many times $t_j$ with 
\[
0\le t_1<\dots <t_N\le L.
\]
We let  $t_0=-\varepsilon$ and $t_{N+1}=L+\varepsilon$.  By \cite[Lemma 3.5]{DKuLS_2016},  there exists an open cover $\{(U_j,\kappa_j)\}_{j=0}^{N+1}$ of $\gamma([-\varepsilon,L+\varepsilon])$ consisting of coordinate neighborhoods having the following properties: 
\begin{itemize}
\item[(i)] $\kappa_j(U_j)=I_j\times B$, where $I_j$ are open intervals and $B=B(0,\delta')$ is an open ball in $\R^{n-2}$.  Here $\delta'>0$  can be taken arbitrarily small and the same for each $U_j$, 
\item[(ii)] $\kappa_j(\gamma(t))=(t,0)$ for each $t\in I_j$,
\item[(iii)] $t_j$ only belongs to $I_j$ and $\overline{I_j}\cap \overline{I_k}=\emptyset$ unless $|j-k|\le 1$,
\item[(iv)] $\kappa_j=\kappa_k$ on $\kappa_j^{-1}((I_j\cap I_k)\times B)$.
\end{itemize}
To construct the quasimode $v_s$ globally, we first find a function  $v_s^{(0)}=e^{is\varphi^{(0)}}a^{(0)}$, $a^{(0)}=h^{-\frac{(n-2)}{4}}a_0^{(0)}\chi$, in $U_0$ as above. Choose some $t_0'$ with $\gamma(t_0')\in U_0\cap U_1$. To construct the phase $\varphi^{(1)}$ in $U_1$, we solve the Riccati equation \eqref{eq_Riccati} with the initial condition $H^{(1)}(t_0')=H^{(0)}(t_0')$. Continuing in this way, we obtain the phases $\varphi^{(0)}, \varphi^{(1)}, \dots, \varphi^{(N+1)}$ such that 
$\varphi^{(j)}=\varphi^{(j+1)}$ on $U_j\cap U_{j+1}$. In a similar way, by solving ODE in \eqref{eq_200_12_0} with prescribed initial conditions we get $\phi^{(0)},\dots, \phi^{(N+1)}$, and therefore, in view of \eqref{eq_guass_type1_amplitude} 
we obtain $a_0^{(0)}, a_0^{(1)}, \dots, a_0^{(N+1)}$, and hence, we construct the amplitude of the first type globally. 

To construct the amplitude of the second type, we need to solve the inhomogeneous  $\bar\p$--type equations 
\eqref{eq_200_15_2}.  To that end,  we first find $a_0^{(0)}$ and $a_0^{(1)}$ which are solutions of 
 \eqref{eq_200_15_2} on $\tilde J\times I_0$ and on $\tilde J\times I_1$, respectively. 
Here $\tilde J\subset \R$ is a bounded open interval. Then  we see that $ e^{\phi^{(1)}}a_{0}^{(1)}-e^{\phi^{(0)}} a_{0}^{(0)}$ is holomorphic on $\tilde J\times (I_0\cap I_1)$. By \cite[Example 3.25]{BG_book}, there are holomorphic functions $g_1, g_0$  on $\tilde J\times I_1$ and  $\tilde J\times I_0$, respectively, such that $e^{\phi^{(1)}}a_{0}^{(1)}-e^{\phi^{(0)}} a_{0}^{(0)}=g_0-g_1$ on $\tilde J\times (I_0\cap I_1)$.  Thus, modifying $a_{0}^{(0)}$ and $ a_{0}^{(1)}$, we can always arrange so that $ a_{0}^{(0)}= a_{0}^{(1)}$ on $\tilde J\times (I_0\cap I_1)$.  Proceeding in the same way, we can find $ a_{0}^{(2)}, \dots,  a_{0}^{(N+1)}$ so that $ a_{0}^{(j)}= a_{0}^{(j+1)}$ on  $\tilde J\times (I_j\cap I_{j+1})$, and hence,  we construct the amplitude of the second type globally.

Thus, we obtain  the quasimodes $v_s^{(0)},\dots, v_s^{(N+1)}$ such that 
\begin{equation}
\label{eq_equal_quasi}
v_s^{(j)}(x_1,\cdot)=v_s^{(j+1)}(x_1,\cdot)\quad \text{in}\quad  U_j\cap U_{j+1},
\end{equation}
for all $x_1$. Let $\chi_j=\chi_j(t)\in C^\infty_0(I_j)$ be such that $\sum_{j=0}^{N+1} \chi_j=1$ near $[-\varepsilon,L+\varepsilon]$, and define our quasimode $v$ globally by
 \[
 v_s=\sum_{j=0}^{N+1} \chi_j v_s^{(j)}. 
 \]

Let us next give a local description of the quasimode $v_s$ near self-intersecting points of the geodesic $\gamma$ and near the other points of $\gamma$.  To that end, let  $p_1,\dots,p_R\in M_0$ be the distinct points where the geodesic self-intersects, and let $0\le t_1<\dots<t_{R'}$ be the times of self-intersections. Let $V_1,\dots, V_R$ be small neighborhoods in $\hat M_0$ around $p_j$, $j=1,\dots, R$. Then choosing $\delta'$ small enough we obtain an open cover in $\hat M_0$, 
\begin{equation}
\label{eq_open_cover_sup}
\supp(v_s(x_1,\cdot ))\cap M_0\subset (\cup_{j=1}^R V_j)\cup (\cup_{k=1}^S W_k),
\end{equation}
where in each $V_j$, the quasimode is a finite sum,
\begin{equation}
\label{eq_rep_vs_1}
v_s(x_1,\cdot)|_{V_j}=\sum_{l: \gamma(t_l)=p_j} v_s^{(l)}(x_1,\cdot),
\end{equation}
and in each $W_k$ (where there are no self-intersecting points), in view of \eqref{eq_equal_quasi}, there is some $l(k)$ so that the quasimode is given by 
\begin{equation}
\label{eq_rep_vs_2}
v_s(x_1,\cdot)|_{W_k}=v_s^{l(k)}(x_1,\cdot).
\end{equation}
We also have 
\[
\supp(v_s)\cap M\subset (\cup_{j=1}^R  \tilde J\times V_j)\cup (\cup_{k=1}^S \tilde J\times W_k),
\]
where $\tilde J\subset\R$ is a bounded open interval. 

Finally, the bounds in \eqref{eq_prop_gaussian_1}  follows from the bounds \eqref{eq_200_17},  \eqref{eq_200_18}, \eqref{eq_200_27}, and the representations \eqref{eq_rep_vs_1} and \eqref{eq_rep_vs_2} of $v$.

\textit{Step 6. Construction of the Gaussian beam quasimodes $w_s$.}

Now look for a Gaussian beam quasimode for the operator $e^{-sx_1}(-h^2\Delta_g)^2 e^{sx_1}$ in the form
\begin{equation}
\label{eq_200_30_w_s}
w_s(x_1,t,y) = e^{is\varphi(t,y)}b(x_1,t,y;s),
\end{equation}
where  $\varphi \in C^\infty(U)$ is  the phase function given by \eqref{eq_gauss_9}, and $b\in C^{\infty}(\R\times U)$ is an amplitude, which we shall proceed to determine. To that end, first similar to \eqref{eq_200_2}, we get 
\begin{equation}
\label{eq_200_30}
e^{-sx_1} (-h^2 \Delta_g)^2 e^{sx_1} w_s=e^{is\varphi}h^4\big(-\Delta_g -s(2 \langle \nabla_g \tilde{\tilde \varphi}, \nabla_g \cdot \rangle_g +\Delta_g \tilde{\tilde \varphi})-s^2\langle \nabla_g \tilde{\tilde \varphi}, \nabla_g \tilde{\tilde \varphi}\rangle_g\big)^2b,
\end{equation}
where 
\begin{equation}
\label{eq_200_31}
\tilde{\tilde\varphi}(x_1,t,y)=x_1+i\varphi(t,y).  
\end{equation}
With $\varphi$ given by \eqref{eq_gauss_9}, we have 
\[
\langle \nabla_g \tilde{\tilde \varphi}, \nabla_g \tilde{\tilde \varphi}\rangle_g=\mathcal{O}(|y|^3),
\]
as $y\to 0$. We thus look for the smooth  amplitude $b=b(x_1,x')$ satisfying the transport equation, 
\begin{equation}
\label{eq_200_31_tilde_L_trans}
\tilde L^2b=\mathcal{O}(|y|),
\end{equation}
where 
\begin{equation}
\label{eq_200_31_tilde_L}
\tilde L=2 \langle \nabla_g \tilde{\tilde \varphi}, \nabla_g \cdot \rangle_g +\Delta_g \tilde{\tilde \varphi}.
\end{equation}
Let us simplify the operator $\tilde L$. First using \eqref{eq_200_31}, we get 
\begin{equation}
\label{eq_200_32}
\langle \nabla_g \tilde{\tilde \varphi}, \nabla_g \cdot \rangle_g= \frac{1}{c}(\p_{x_1}+ig_0^{-1}(x')\varphi'_{x'}\cdot\p_{x'} ), 
\end{equation}
\begin{equation}
\label{eq_200_33}
\Delta_g \tilde \varphi=\Delta_g x_1+i \Delta_g \varphi(x').
\end{equation}
Hence, using \eqref{eq_200_32}, \eqref{eq_200_33}, \eqref{eq_200_7}, and \eqref{eq_200_8}, the operator $\tilde L$ given by \eqref{eq_200_31_tilde_L} becomes 
\begin{equation}
\label{eq_200_34}
\tilde L=\frac{2}{c}(\p_{x_1}+ig_0^{-1}(x')\varphi'_{x'}\cdot\p_{x'}) +\bigg(\frac{n}{2}-1\bigg) \frac{1}{c^2} \partial_{x_1}c+\frac{i}{c} \Delta_{g_0} \varphi + \bigg(\frac{n}{2} -1\bigg) \frac{i}{c^2}  \langle \nabla_{g_0}c, \nabla_{g_0} \varphi \rangle_{g_0}.
\end{equation}

Using \eqref{eq_200_11}, \eqref{eq_200_11_2}, \eqref{eq_guassian_inner}, the operator $\tilde L$ in 
\eqref{eq_200_34} becomes
\begin{equation}
\begin{aligned}
\label{eq_gauss_14_tilde}
\tilde L = &=\frac{2}{c(x_1,t,0)}\bigg[ \partial_{x_1} +i \partial_t  + iH(t)y\cdot\partial_y + (\partial_{x_1}+i\partial_t)\log c(x_1,t,0)^{\frac{n}{4}-\frac{1}{2}}\\
&  +\frac{i}{2}\tr H(t)+\mathcal{O}(|y|)+\mathcal{O}(|y|)(\partial_{x_1},\partial_t) + \mathcal{O}(|y|^2) \partial_y \bigg].
\end{aligned}
\end{equation}
We look for  the amplitude b in the form
\begin{equation}
\label{eq_200_35}
b(x_1,t,y) = h^{-\frac{(n-2)}{4}}b_0(x_1,t)\chi \bigg( \frac{y}{\delta'} \bigg), 
\end{equation}
where $b_0(\cdot, \cdot) \in C^\infty(\R\times  \{t: |t-t_0|<\delta\} )$ is independent of $y$, and in view of \eqref{eq_200_31_tilde_L_trans}, $b_0$ should satisfy
\begin{equation}
\label{eq_200_36}
\tilde{L}^2b_0 = \mathcal{O}(|y|),\quad y\to 0. 
\end{equation}
It follows from \eqref{eq_200_34} that 
\begin{equation}
\label{eq_gauss_14_tilde_L_0}
\tilde L=\frac{2}{c(x_1,t,0)}(\tilde L_0+\tilde R),
\end{equation}
where 
\begin{equation}
\label{eq_gauss_17_tilde}
\tilde L_0 =  (\partial_{x_1} + i \partial_t )+(\partial_{x_1}+i\partial_t)\log c(x_1,t,0)^{\frac{n}{4}-\frac{1}{2}} +\frac{i}{2}\tr H(t),
\end{equation}
and 
\begin{equation}
\label{eq_gauss_17_R_tilde}
\tilde R=iH(t)y\cdot\partial_y+\mathcal{O}(|y|)+\mathcal{O}(|y|)(\partial_{x_1},\partial_t) + \mathcal{O}(|y|^2) \partial_y. 
\end{equation}
In contrast to the construction of the Gaussian beam quasimodes $v_s$, we shall only need amplitudes of the first type. To construct such amplitudes, we note that as  $b_0$ is independent of $y$, if $b_0$ solves the following equation 
\begin{equation}
\label{eq_gauss_16_tilde}
\tilde L_0b_0 = 0,
\end{equation}
then $b_0$ satisfies \eqref{eq_200_36}. To find a solution to \eqref{eq_gauss_16_tilde}, we note that 
\begin{equation}
\label{eq_200_12_tilde}
\tilde L_0 = e^{-\tilde \phi(x_1,t)}(\partial_{x_1}+i\partial_t)e^{\tilde \phi(x_1,t)},
\end{equation}
where $\tilde \phi(x_1,t)$ is given by 
\begin{equation}
\label{eq_guass_type1_amplitude_tilde_0}
\tilde \phi(x_1,t)=\log c(x_1,t,0)^{\frac{n}{4}-\frac{1}{2}}+F(t), \quad \partial_t F(t)=\frac{1}{2}\tr H(t).
\end{equation}
We solve \eqref{eq_gauss_16_tilde} by taking 
\begin{equation}
\label{eq_guass_type1_amplitude_tilde}
b_0 = e^{-\tilde \phi} = c(x_1,t,0)^{\frac{1}{2}-\frac{n}{4}}e^{-F(t)}.
\end{equation}

Proceeding further as in the construction of the quasimode $v_s$ above, we obtain the quasimode $w_s \in C^\infty(M)$ such that \eqref{eq_prop_gaussian_2} holds. 

\end{proof}

We shall need the following result. 
\begin{prop}
\label{prop_computations_gaussian_1}
Let $X\in C(M,TM)$ be a complex vector field, let $\psi\in C(M_0)$, and let $x'_1\in \R$. Then there exist the Gaussian beam quasimodes $v_s$ and $w_s$ given by Proposition \ref{prop_Gaussian_beams} such that $v_s$ is obtained using  amplidutes of the first type and we have 
\begin{equation}
\label{eq_prop_gaussian_3}
\lim_{h\to 0} \int_{\{x'_1\}\times M_0} v_s \overline{w_s} \psi dV_{g_0}= \int_0^L e^{-2\lambda t}  c(x_1,\gamma(t))^{1-\frac{n}{2}} \psi(\gamma(t)) dt,
\end{equation}
and
\begin{equation}
\label{eq_prop_gaussian_4}
\lim_{h\to 0} h\int_{\{x'_1\}\times M_0} X(v_s)\overline{w_s}\psi dV_{g_0} 
=i\int_0^L  X_t(x_1',\gamma(t)) e^{-2\lambda t} c(x_1,\gamma(t))^{1-\frac{n}{2}} \psi(\gamma(t)) dt.
\end{equation}
Here $X_t(x_1',\gamma(t))=\langle X(x_1',\gamma(t)),(0,\dot{\gamma}(t))\rangle_g$.
\end{prop}

\begin{proof}

\textit{Step 1. Proof of  the estimate \eqref{eq_prop_gaussian_3}.}

Let $\psi\in C(M_0)$, $x'_1\in \R$. 
Using a partition of unity, in view of \eqref{eq_open_cover_sup}, it suffices to establish \eqref{eq_prop_gaussian_3} for $\psi$ having compact support in one of the sets $V_j$ or   $W_k$.   First, assume that  $\psi\in C_0 ( M_0)$, $\supp(\psi)\subset W_k$. Thus, in view of \eqref{eq_rep_vs_2}, \eqref{eq_200_16}, \eqref{eq_200_30_w_s}, \eqref{eq_200_35},  on $\supp(\psi)$, we have 
\begin{equation}
\label{eq_gauss_vs_ws}
v_s=e^{is\varphi} h^{-\frac{(n-2)}{4}}a_0 (x'_1,t)\chi\bigg(\frac{y}{\delta'}\bigg),\quad w_s=e^{is\varphi} h^{-\frac{(n-2)}{4}}b_0 (x'_1,t)\chi\bigg(\frac{y}{\delta'}\bigg).
\end{equation}
 To proceed, we shall need the following consequence of \eqref{eq_gauss_3_metric},  
\begin{equation}
\label{eq_gauss_|g_0|}
 |g_0|^{1/2}=1+\mathcal{O}(|y|^2),
\end{equation}
as well as 
\begin{equation}
\label{eq_200_38}
is\varphi -i\overline{s}\overline{\varphi}=-2\frac{1}{h}\text{Im}\varphi-2\lambda\text{Re}\varphi. 
\end{equation}

Using \eqref{eq_gauss_vs_ws}, \eqref{eq_gauss_|g_0|},  \eqref{eq_200_38},  \eqref{eq_gauss_9}, we get 
\begin{equation}
\label{eq_gauss_25}
\begin{aligned}
&\int_{\{x'_1\}\times M_0} v_s \overline{w_s} \psi dV_{g_0}\\
&=\int_{0}^L \int_{\R^{n-2}} e^{-2\frac{1}{h}\text{Im}\varphi} e^{-2\lambda\text{Re}\varphi} h^{-\frac{(n-2)}{2}} a_0(x'_1,t)\overline{b_0(x'_1,t)} \chi^2\bigg(\frac{y}{\delta'}\bigg) \psi(t,y) |g_0|^{\frac{1}{2}} dy dt\\
&=\int_0^L \int_{\R^{n-2}} e^{-\frac{1}{h}\text{Im} H(t)y\cdot y}  e^{-2\lambda t} e^{\lambda \mathcal{O}(|y|^2)} h^{-\frac{(n-2)}{2}} a_0(x'_1,t)\overline{b_0(x'_1,t)}\chi^2\bigg(\frac{y}{\delta'}\bigg)\\
& \psi(t,y)(1+\mathcal{O}(|y|^2)) dy dt.
\end{aligned}
\end{equation}
Making the change of variable $y=h^{1/2}\tilde y$ in  \eqref{eq_gauss_25}, we obtain that  
\begin{equation}
\label{eq_gauss_25_1}
\begin{aligned}
\int_{\{x'_1\}\times M_0} v_s \overline{w_s} \psi dV_{g_0}=\int_0^L\int_{\R^{n-2}} e^{-\text{Im} H(t)\tilde y\cdot\tilde y}e^{-2\lambda t}e^{\lambda h\mathcal{O}(|\tilde y|^2)}a_0(x_1',t)\overline{b_0(x_1',t)}\\
\chi^2\bigg(\frac{h^{1/2}\tilde y}{\delta'}\bigg)\psi(t,h^{1/2}\tilde y)(1+h\mathcal{O}(|\tilde y|^2))dtd\tilde y.
\end{aligned}
\end{equation}

Using that 
\begin{equation}
\label{eq_gauss_27}
\int_{\R^{n-2}} e^{-\text{Im} H(t) y\cdot y} dy=\frac{\pi^{(n-2)/2}}{\sqrt{\det(\text{Im} H(t))}},
\end{equation}
and the dominated covergence theorem, we get from 
\eqref{eq_gauss_25_1} that 
\begin{equation}
\label{eq_gauss_26}
\begin{aligned}
& \lim_{h\to 0}\int_{\{x'_1\}\times M_0} v_s \overline{w_s} \psi dV_{g_0}\\
&=\int_0^L e^{-2\lambda t}  a_0(x'_1,t)\overline{b_0(x'_1,t)} \psi(t,0) \int_{\R^{n-2}} e^{-\text{Im} H(t) y\cdot y} dydt\\
&=
\int_0^L e^{-2\lambda t}  a_0(x'_1,t)\overline{b_0(x'_1,t)}\frac{\pi^{(n-2)/2}}{\sqrt{\det(\text{Im} H(t))}} \psi(t,0) dt.
\end{aligned}
\end{equation}
Let us proceed to simplify the expression in \eqref{eq_gauss_26} in the case when $a_0$ is the amplitude of the first type, i.e.  $a_0$ be given by \eqref{eq_guass_type1_amplitude},  and $b_0$ be given by 
\eqref{eq_guass_type1_amplitude_tilde}. Then 
\begin{equation}
\label{eq_gauss_26_100}
a_0(x'_1,t)\overline{b_0(x'_1,t)}\frac{\pi^{(n-2)/2}}{\sqrt{\det(\text{Im} H(t))}}=c(x_1,t,0)^{1-\frac{n}{2}}e^{-(G(t)+\overline{F(t)})}\frac{\pi^{(n-2)/2}}{\sqrt{\det(\text{Im} H(t))}}.
\end{equation}
Now it follows from \eqref{eq_guass_type1_amplitude} and
\eqref{eq_guass_type1_amplitude_tilde_0} that 
\begin{equation}
\label{eq_gauss_26_101}
G(t)+\overline{F(t)}=G(t_0)+\overline{F(t_0)}+\int_{t_0}^t \text{tr}\, \text{Re}(H(s))ds.
\end{equation}
Using \eqref{eq_gauss_26_101} and the following property of solutions of the matrix Riccati equation
\cite[Lemma 2.58]{KKL_book}, 
\[
\det\, (\text{Im} H(t))=\det \, (\text{Im} H(t_0)) e^{-2\int_{t_0}^t\text{tr}\, \text{Re}(H(s))ds},
\]
we see that 
\begin{equation}
\label{eq_gauss_26_102}
e^{-(G(t)+\overline{F(t)})}\frac{\pi^{(n-2)/2}}{\sqrt{\det(\text{Im} H(t))}}=e^{-(G(t_0)+\overline{F(t_0)})}\frac{\pi^{(n-2)/2}}{\sqrt{\det(\text{Im} H(t_0))}}
\end{equation}
is a constant in $t$.  To fix this constant, when constructing the amplitude $a_0$ and $b_0$, specifically, when  solving \eqref{eq_guass_type1_amplitude} and
\eqref{eq_guass_type1_amplitude_tilde_0} in $U_0$, we choose initial conditions for $G$ and $F$ so that the constant in \eqref{eq_gauss_26_102} is equal to $1$. With this choice, it follows from \eqref{eq_gauss_26},  \eqref{eq_gauss_26_100}, \eqref{eq_gauss_26_102} that 
\begin{equation}
\label{eq_gauss_26_103}
 \lim_{h\to 0}\int_{\{x'_1\}\times M_0} v_s \overline{w_s} \psi dV_{g_0}=
\int_0^L e^{-2\lambda t}  c(x_1,t,0)^{1-\frac{n}{2}} \psi(t,0) dt.
\end{equation}
This completes the proof of \eqref{eq_prop_gaussian_3} in the case when $\supp(\psi)\subset W_k$.

Let us now establish \eqref{eq_prop_gaussian_3} when   $\supp(\psi) \subset V_j$.  Here  on $\supp(\psi)$ we have
\begin{equation}
\label{eq_gaus_v_s_w_s_sum_-10} 
v_s=\sum_{l: \gamma(t_l)=p_j} v_s^{(l)}, \quad w_s=\sum_{l: \gamma(t_l)=p_j} w_s^{(l)},
\end{equation}
and hence, 
\begin{equation}
\label{eq_gaus_v_s_w_s_sum} 
v_s\overline{w_s}=\sum_{l: \gamma(t_l)=p_j} v_s^{(l)}\overline{w_s^{(l)}}+ \sum_{l\ne l',\gamma(t_l)=\gamma(t_{l'})=p_j} v_s^{(l)}\overline{w_s^{(l')}}.
\end{equation}
We shall use a non-stationary phase argument as in \cite[end of proof Proposition 3.1]{DKuLS_2016} to show that the contribution of the mixed terms vanishes  in the limit $h\to 0$, i.e. if $l\ne l'$,
\begin{equation}
\label{eq_gauss_33}
\lim_{h \to 0}   \int_{\{x'_1\}\times M_0} v^{(l)}_s \overline{w^{(l')}_s} \psi dV_{g_0}= 0.
\end{equation}
In doing so, write
\[
v_s^{(l)}=e^{i \frac{1}{h}\text{Re}\, \varphi^{(l)}}p^{(l)}, \quad p^{(l)}=e^{-\lambda \text{Re}\, \varphi^{(l)}} e^{-s \text{Im}\, \varphi^{(l)}}a^{(l)},
\]
and
\[
w_s^{(l')}=e^{i \frac{1}{h}\text{Re}\, \varphi^{(l')}}q^{(l')}, \quad q^{(l')}=e^{-\lambda \text{Re}\, \varphi^{(l')}} e^{-s \text{Im}\, \varphi^{(l')}}b^{(l')},
\]
and therefore, 
\begin{equation}
\label{eq_gauss_34}
v^{(l)}_s \overline{w^{(l')}_s} =e^{i\frac{1}{h}\phi} p^{(l)} \overline{q^{(l')}},
\end{equation}
where 
\[
\phi= \Re \varphi^{(l)}-\Re \varphi^{(l')}.
\]
Thus, in view of \eqref{eq_gauss_33} and \eqref{eq_gauss_34} we shall show that for $l\ne l'$,
\begin{equation}
\label{eq_gauss_35}
\lim_{h \to 0}   \int_{\{x'_1\}\times M_0} e^{i \frac{1}{h}\phi} p^{(l)} \overline{q^{(l')}}  \psi dV_{g_0}= 0.
\end{equation}
Since $\p_t\varphi^{(l)}(t,0)=\p_t\varphi^{(l')}(t,0)=1$ and the geodesic intersects itself transversally, as explained in \cite[Lemma 7.2]{Kenig_Salo_APDE_2013},   we see that $d\phi(p_j) \ne 0$.  
By decreasing the set $V_j$ if necessary, we may assume that  $d\phi\ne 0$ in $V_j$. 

To prove \eqref{eq_gauss_35}, we shall integrate by parts and in doing so, we let $\varepsilon>0$ be fixed, and decompose 
$
\psi=\psi_1+\psi_2,
$
where $\psi_1\in C^\infty(M_0)$, $\supp(\psi_1)\subset V_j$ and  and $\|\psi_2\|_{L^\infty(V_j\cap M_0)}\le \varepsilon$.  Notice that $\psi$ may be nonzero on $\p M_0$.  We have
\begin{equation}
\label{eq_gauss_36}
\bigg| \int_{\{x'_1\}\times M_0} e^{i \frac{1}{h}\phi} p^{(l)} \overline{q^{(l')}}  \psi_2 dV_{g_0}\bigg| \le \| v^{(l)}_s\|_{L^2}\| w^{(l)}_s\|_{L^2}\|\psi_2\|_{L^\infty}\le \mathcal{O} (\varepsilon). 
\end{equation}
For the smooth part $\psi_1$, we integrate by parts using that 
\[
e^{i\frac{1}{h}\phi}=\frac{h}{i} L(e^{i\frac{1}{h}\phi}), \quad L=\frac{1}{|d\phi|^{2}}\langle  d\phi, d\cdot \rangle_{g_0}.
\]
We have
\begin{equation}
\label{eq_gauss_37}
\begin{aligned}
 \int_{\{x'_1\}\times M_0} e^{i\frac{1}{h}\phi} p^{(l)} \overline{q^{(l')}}  \psi_1 dV_{g_0} =&\int_{\{x'_1\}\times (V_j\cap \p M_0)} h\frac{\p_\nu \phi}{i |d\phi|^2} e^{i\frac{1}{h}\phi} p^{(l)} \overline{q^{(l')}} \psi_1 dS \\
 &+h\frac{1}{i} \int_{\{x'_1\}\times M_0} e^{i\frac{1}{h}\phi} L^t( p^{(l)} \overline{q^{(l')}}  \psi_1) dV_{g_0}, 
\end{aligned}
\end{equation}
where $L^t=-L-\div L$ is the transpose of $L$. 

In view of \eqref{eq_v_s_on_boundary}, the boundary term is of $\mathcal{O}(h)$ as $h \to 0$. To estimate the second term in the right hand side of \eqref{eq_gauss_37}, we recall that 
\begin{align*}
p^{(l)} \overline{q^{(l')}}=e^{-\lambda (\text{Re}\, \varphi^{(l)}+\text{Re}\, \varphi^{(l')})} e^{-i\lambda (\text{Im}\, \varphi^{(l)}-\text{Im}\, \varphi^{(l')})} 
e^{-\frac{1}{h} (\text{Im}\, \varphi^{(l)}+\text{Im}\, \varphi^{(l')})} 
h^{-\frac{(n-2)}{2}}\\
 a_0^{(l)}(x_1',t)\overline{b_0^{(l')}(x_1',t)}\chi^2 \bigg(\frac{y}{\delta'}\bigg).
\end{align*}
This shows that to bound the second term in the right hand side of \eqref{eq_gauss_37}, it is enough to analyze the contributions occurring when differentiating 
\[
e^{-\frac{1}{h}(\text{Im}\, \varphi^{(l)}+\text{Im}\, \varphi^{(l')})},
\]
as all the other contributions are  of $\mathcal{O}(h)$, as $h\to 0$.  

As in \cite{DKuLS_2016}, using \eqref{eq_200_3_2}, we have
\[
|L (e^{-\frac{1}{h}(\text{Im}\, \varphi^{(l)}+\text{Im}\, \varphi^{(l')})} )|\le  \mathcal{O}(h^{-1}) |d (\text{Im}\, \varphi^{(l)}+\text{Im}\, \varphi^{(l')}) |e^{-\frac{1}{h} d|y|^2}\le \mathcal{O}
(h^{-1} |y|)e^{-\frac{1}{h} d|y|^2},
\]
which shows that the corresponding contribution to the second term  in the right hand side of \eqref{eq_gauss_37} is of $\mathcal{O}(h^{1/2})$. 
This shows that the integral in the left hand side of \eqref{eq_gauss_37}  goes to $0$ as $h\to 0$, and this together with \eqref{eq_gauss_36} establishes \eqref{eq_gauss_33}. 

Using \eqref{eq_gauss_26_103} for each of the factors $v_s^{(l)}\overline{w_s^{(l)}}$ in \eqref{eq_gaus_v_s_w_s_sum}, we get
\[
\lim_{h\to 0}   \int_{\{x'_1\}\times M_0} v^{(l)}_s \overline{w^{(l)}_s} \psi dV_{g_0}=\int_{I_l} e^{-2\lambda t}  c(x_1,t,0)^{1-\frac{n}{2}} \psi(t,0) dt.
\]
Summing over $I_l$, appearing in the Fermi coordinates,  such that $t_l\in I_l$ and $\gamma(t_l)=p_j$, we get \eqref{eq_prop_gaussian_3} when   $\supp(\psi) \subset V_j$, and hence, in general.   

\textit{Step 2. Establishing the estimate \eqref{eq_prop_gaussian_4}.}

Let  $X\in C(M,TM)$ be complex vector field,  $\psi\in C(M_0)$ and $x'_1\in \R$. Using a partition of unity, it is enough to verify \eqref{eq_prop_gaussian_4} in the following two cases: $\supp(\psi)\subset W_k$ and $\supp(\psi)\subset V_j$. Assume first that  $\supp(\psi)\subset W_k$. Using \eqref{eq_gauss_vs_ws}, we get 
\begin{equation}
\label{eq_gauss_38}
h\int_{\{x'_1\}\times M_0} X(v_s)\overline{w_s}\psi dV_{g_0}=I_{1,1}+I_{1,2}+I_2,
\end{equation}
where 
\begin{equation}
\label{eq_gauss_38_I_1_1}
I_{1,1}=  \int_{\{x'_1\}\times M_0} i X(\varphi)v_s \overline{w_s} \psi dV_{g_0},
\end{equation}
\begin{equation}
\label{eq_gauss_38_I_1_2}
I_{1,2}=- h \int_{\{x'_1\}\times M_0} \lambda X(\varphi)v_s \overline{w_s} \psi dV_{g_0},
\end{equation}
\begin{equation}
\label{eq_gauss_38_I_2}
I_2= h\int_{\{x'_1\}\times M_0} h^{-\frac{(n-2)}{4}}e^{is\varphi} X(a_0\chi)\overline{w_s} \psi dV_{g_0}.
\end{equation}
Using \eqref{eq_prop_gaussian_1} and \eqref{eq_prop_gaussian_2}, we have 
\begin{equation}
\label{eq_200_50}
\begin{aligned}
&|I_{1,2}|\le \mathcal{O}(h)\|v_s(x',\cdot)\|_{L^2(M_0)}\|w_s(x_1',\cdot)\|_{L^2(M_0)}=\mathcal{O}(h),\\
&|I_2|\le \mathcal{O}(h)\|e^{is\varphi}h^{-\frac{(n-2)}{4}}\|_{L^2(\{|y|\le \delta'/2\})}\|w_s(x_1',\cdot)\|_{L^2(M_0)}=\mathcal{O}(h).
\end{aligned}
\end{equation}
Let us now compute $\lim_{h\to 0}I_{1,1}$. To that end, we write
\begin{equation}
\label{eq_200_51}
X=X_1\p_{x_1}+X_t\p_t+X_y\cdot \p_y,\quad x=(x_1,t,y).
\end{equation}
Using \eqref{eq_gauss_9}, we get
\begin{equation}
\label{eq_200_52}
\p_t\varphi=1+\mathcal{O}(|y|^2),\quad \p_y\varphi=\mathcal{O}(|y|).
\end{equation}
As $X$ is continuous, it follows from \eqref{eq_200_51} and \eqref{eq_200_52} that 
\begin{equation}
\label{eq_200_53}
X(\varphi)=(X_t(x_1,t,0)+o(1))(1+\mathcal{O}(|y|^2))+\mathcal{O}(|y|)=X_t(x_1,t,0)+o(1),
\end{equation}
as $y\to 0$, uniformly in $x_1$ and $t$. 
Using \eqref{eq_200_53}, as in \eqref{eq_gauss_25}, we obtain from \eqref{eq_gauss_38_I_1_1} that 
\begin{equation}
\label{eq_200_54}
\begin{aligned}
I_{1,1}=\int_0^L\int_{\R^{n-2}}i (X_t(x_1',t,0)+o(1))h^{-\frac{(n-2)}{2}}e^{-\frac{1}{h}\text{Im}H(t)y\cdot y} e^{-2\lambda t}e^{\lambda \mathcal{O}(|y|^2)}\\
a_0(x_1',t)\overline{b_0(x_1',t)}\chi^2\bigg(\frac{y}{\delta'}\bigg)\psi(t,y)(1+\mathcal{O}(|y|^2))dydt.
\end{aligned}
\end{equation}
We first observe that 
\begin{equation}
\label{eq_200_54_1}
\lim_{h\to 0} I_{1,1,2}=0,
\end{equation}
uniformly in $x_1'$ and $t$,
where
\begin{align*}
I_{1,1,2}=\int_{\R^{n-2}} g(x_1', t,y) dy, \quad 
g(x_1', t,y) = o(1)h^{-\frac{(n-2)}{2}}e^{-\frac{1}{h}\text{Im}H(t)y\cdot y} e^{-2\lambda t}\\
e^{\lambda \mathcal{O}(|y|^2)}
a_0(x_1',t)\overline{b_0(x_1',t)}\chi^2\bigg(\frac{y}{\delta'}\bigg)\psi(t,y)
(1+\mathcal{O}(|y|^2)).
\end{align*}
Indeed, let $\varepsilon>0$ and let $\delta>0$ be such that $|o(1)|\le \varepsilon$ when $|y|\le \delta$. Then 
\begin{align*}
|I_{1,1,2}|&\le \bigg|\int_{|y|\le \delta} g(x_1', t,y)dy\bigg| +\bigg|\int_{|y|\ge \delta} g(x_1', t,y)dy\bigg| \\
&\le \varepsilon \mathcal{O}(1)\bigg|\int_{\R^{n-2}} h^{-\frac{(n-2)}{2}}e^{-\frac{1}{h}\text{Im}H(t)y\cdot y}  dy\bigg| +\mathcal{O}(e^{-d\delta^2/h})\le \varepsilon \mathcal{O}(1)+\mathcal{O}(e^{-d\delta^2/h}),
\end{align*}
showing \eqref{eq_200_54_1}. 

Using \eqref{eq_200_54_1}, making the change of variables $y=h^{1/2}\tilde y$ in \eqref{eq_200_54}, using the dominated convergence theorem, and \eqref{eq_gauss_27}, we get 
\begin{equation}
\label{eq_200_55}
\lim_{h\to 0} I_{1,1}=i\int_0^L  X_t(x_1',t,0) e^{-2\lambda t} a_0(x_1',t)\overline{b_0(x_1',t)} \psi(t,0)\frac{\pi^{(n-2)/2}}{\sqrt{\det(\text{Im}H(t))}}dt.
\end{equation}
It follows from \eqref{eq_gauss_38} with the help of \eqref{eq_200_50} and \eqref{eq_200_55} that 
\begin{equation}
\label{eq_200_56}
\begin{aligned}
\lim_{h\to 0} h&\int_{\{x'_1\}\times M_0} X(v_s)\overline{w_s}\psi dV_{g_0} \\
&=i\int_0^L  X_t(x_1',t,0) e^{-2\lambda t} a_0(x_1',t)\overline{b_0(x_1',t)} \psi(t,0)\frac{\pi^{(n-2)/2}}{\sqrt{\det(\text{Im}H(t))}}dt.
\end{aligned}
\end{equation}
When $a_0$ is the amplitude of the first type, i.e.  $a_0$ be given by \eqref{eq_guass_type1_amplitude},  and $b_0$ be given by  \eqref{eq_guass_type1_amplitude_tilde},  using \eqref{eq_gauss_26_100}, \eqref{eq_gauss_26_102}, we get from 
\eqref{eq_200_56}  that 
\begin{equation}
\label{eq_200_57}
\lim_{h\to 0} h\int_{\{x'_1\}\times M_0} X(v_s)\overline{w_s}\psi dV_{g_0} 
=i\int_0^L  X_t(x_1',t,0) e^{-2\lambda t} c(x_1,t,0)^{1-\frac{n}{2}} \psi(t,0) dt.
\end{equation}
This establishes \eqref{eq_prop_gaussian_4} when $\supp(\psi)\subset W_k$. 

Assume now that $\supp(\psi)\subset V_j$, and therefore, on $\supp(\psi)$, $v_s$ and $w_s$ are given by \eqref{eq_gaus_v_s_w_s_sum_-10}. Then 
\begin{equation}
\label{eq_200_58}
\begin{aligned}
h\int_{\{x_1'\}\times M_0} X(v_s) \overline{w_s}\psi dV_{g_0}=&h\sum_{l:\gamma(t_l)=p_j}\int_{\{x_1'\}\times M_0} X(v_s^{(l)}) \overline{w_s^{(l)}}\psi dV_{g_0}\\
&+h\sum_{l\ne l':\gamma(t_l)=\gamma(t_{l'})=p_j}\int_{\{x_1'\}\times M_0} X(v_s^{(l)}) \overline{w_s^{(l')}}\psi dV_{g_0}.
\end{aligned}
\end{equation}
As before, we shall show that the mixed terms, i.e. $l\ne l'$, vanish in the limit as $h\to 0$, 
\begin{equation}
\label{eq_200_59_0}
\lim_{h\to 0} h\int_{\{x_1'\}\times M_0} X(v_s^{(l)}) \overline{w_s^{(l')}}\psi dV_{g_0}=0.
\end{equation}
It follows from \eqref{eq_gauss_38}, \eqref{eq_gauss_38_I_1_1},  \eqref{eq_gauss_38_I_1_2},  \eqref{eq_gauss_38_I_2}, \eqref{eq_200_50} that we only have to prove that 
\begin{equation}
\label{eq_200_59}
\lim_{h\to 0} \int_{\{x_1'\}\times M_0} i X(\varphi^{(l)}) v_s^{(l)} \overline{w_s^{(l')}}\psi dV_{g_0}=0.
\end{equation}
Now \eqref{eq_200_59} follows by repeating a non-stationary phase argument as in the proof of \eqref{eq_gauss_33} replacing $\psi$ by $X(\varphi^{(l)}) \psi\in C(M_0)$.  Thus, using \eqref{eq_200_58} and \eqref{eq_200_59}, we see that 
\begin{align*}
\lim_{h\to 0}h\int_{\{x_1'\}\times M_0} X(v_s) &\overline{w_s}\psi dV_{g_0}\\
&=\sum_{l:\gamma(t_l)=p_j} i\int_{I_l}  X_t(x_1',t,0) e^{-2\lambda t} c(x_1,t,0)^{1-\frac{n}{2}} \psi(t,0) dt.
\end{align*}
completing the proof of \eqref{eq_prop_gaussian_4} when $\supp(\psi)\subset V_j$. 
\end{proof}

We shall also need the following result. 
\begin{prop}
\label{prop_computations_gaussian_2}
Let $\psi\in C^1(\R\times M_0)$ be such that $\psi(x_1,\cdot)|_{\p M_0}=0$ and with compact support in $x_1$. Then there exist  Gaussian beam quasimodes $v_s$ and $w_s$ given by Proposition \ref{prop_Gaussian_beams} such that $v_s$ is obtained using  amplitudes of the second type and 
\begin{equation}
\label{eq_prop_gaussian_psi}
\begin{aligned}
\lim_{h\to 0} \bigg[ h \int_\R e^{-2i\lambda x_1}\int_{M_0} (\nabla_g \psi) (v_s) \overline{w_s} c(x_1,x')^{\frac{n}{2}} dV_{g_0}dx_1\\
- \int_\R e^{-2\lambda x_1}\int_{M_0} (\nabla_g \psi)_1 v_s \overline{w_s} c(x_1,x')^{\frac{n}{2}} dV_{g_0}dx_1\bigg]\\
=\int_\R  \int_0^L e^{-2i\lambda (x_1-it)}  \psi(x_1,\gamma(t)) c(x_1,\gamma(t))
dtdx_1.
\end{aligned}
\end{equation}
\end{prop}

\begin{proof}
In view of \eqref{eq_open_cover_sup}, using a partition of unity,  it suffices to check \eqref{eq_prop_gaussian_psi} for $\psi$ such that $\supp(\psi (x_1,\cdot))$ is in one of the sets $V_j$ or $W_k$. Let us first consider the case when $\supp(\psi (x_1,\cdot))\subset W_k$. Thus, on $\supp(\psi (x_1,\cdot))$, $v_s$ and $w_s$ are given by \eqref{eq_gauss_vs_ws} with $a_0$ being an amplitude of type two. To proceed, we note that 
\begin{equation}
\label{eq_prop_gaussian_psi_part_1}
\nabla_g \psi=\frac{1}{c}(\p_{x_1}\psi \p_{x_1}+g_0^{-1}\p_{x'}\psi\cdot\p_{x'}),
\end{equation}
and therefore, using \eqref{eq_gauss_3_metric}, we see that 
\begin{equation}
\label{eq_prop_gaussian_psi_part_1_1}
(\nabla \psi)_t(x_1,t,0)=\frac{\p_{t}\psi(x_1,t,0)}{c(x_1,t,0)}.
\end{equation}

Using \eqref{eq_gauss_vs_ws}, \eqref{eq_prop_gaussian_psi_part_1}, \eqref{eq_prop_gaussian_psi_part_1_1}, a computation similar to that in the proof of Proposition \ref{prop_computations_gaussian_1}, cf. \eqref{eq_gauss_26} and \eqref{eq_200_56}, gives 
\begin{equation}
\label{eq_prop_gaussian_psi_part_2}
\begin{aligned}
I=&\lim_{h\to 0} \bigg[ h \int_\R e^{-2i\lambda x_1}\int_{M_0} (\nabla_g \psi) (v_s) \overline{w_s} c(x_1,x')^{\frac{n}{2}} dV_{g_0}dx_1\\
&- \int_\R e^{-2\lambda x_1}\int_{M_0} (\nabla_g \psi)_1 v_s \overline{w_s} c(x_1,x')^{\frac{n}{2}}  dV_{g_0}dx_1\bigg]\\
&=-\int_\R  \int_0^L e^{-2i\lambda x_1} e^{-2\lambda t} ((\p_{x_1}-i\p_t)\psi(x_1,t,0))
a_0(x_1,t)\overline{b_0(x_1,t)}\\
&\frac{\pi^{(n-2)/2}}{\sqrt{\det(\text{Im} H(t))}}c(x_1,t,0)^{\frac{n}{2}-1}dtdx_1.
\end{aligned}
\end{equation}
When solving the equations \eqref{eq_200_12_0} and \eqref{eq_guass_type1_amplitude_tilde_0} for $G$ and $F$, respectively,  we choose the initial conditions $G(t_0)$ and $F(t_0)$ so that the constant in \eqref{eq_gauss_26_102} is equal to $1$. Then  using \eqref{eq_guass_type1_amplitude_tilde}, \eqref{eq_200_12_0}, \eqref{eq_gauss_26_102},  we see that 
\begin{equation}
\label{eq_prop_gaussian_psi_part_3}
\begin{aligned}
a_0(x_1,t)\overline{b_0(x_1,t)}\frac{\pi^{(n-2)/2}}{\sqrt{\det(\text{Im} H(t))}}c(x_1,t,0)^{\frac{n}{2}-1}
\\
=a_0(x_1,t)c(x_1,t,0)^{\frac{n}{4}-\frac{1}{2}}e^{-\overline{F(t)}}\frac{\pi^{(n-2)/2}}{\sqrt{\det(\text{Im} H(t))}}\\
=a_0(x_1,t)c(x_1,t,0)^{\frac{n}{4}-\frac{1}{2}}e^{G(t)}=a_0(x_1,t)e^{\phi(x_1,t)}
\end{aligned}
\end{equation}
Combining \eqref{eq_prop_gaussian_psi_part_2} and \eqref{eq_prop_gaussian_psi_part_3}, integrating by parts, using the fact that $\psi$ compact support in $x_1$ and $\psi(x_1,\cdot)|_{\p M_0}=0$, and using \eqref{eq_200_15},  we get 
\begin{equation}
\label{eq_prop_gaussian_psi_part_2_1}
\begin{aligned}
I=&-\int_\R  \int_0^L e^{-2i\lambda (x_1-it)} ((\p_{x_1}-i\p_t)\psi(x_1,t,0))
a_0(x_1,t)e^{\phi(x_1,t)}dtdx_1\\
=&\int_\R  \int_0^L e^{-2i\lambda (x_1-it)}  \psi(x_1,t,0)(\p_{x_1}-i\p_t) (a_0(x_1,t)e^{\phi(x_1,t)})
dtdx_1\\
=&\int_\R  \int_0^L e^{-2i\lambda (x_1-it)}  \psi(x_1,t,0)c(x_1,t,0)
dtdx_1.
\end{aligned}
\end{equation}
This completes the proof of \eqref{eq_prop_gaussian_psi} in the case when  $\supp(\psi (x_1,\cdot))\subset W_k$.

Let us know show  \eqref{eq_prop_gaussian_psi} when  $\supp(\psi (x_1,\cdot))\subset V_j$. Then on $\supp(\psi)$, $v_s$ and $w_s$ are given by \eqref{eq_gaus_v_s_w_s_sum_-10}, and we have
\begin{equation}
\label{eq_prop_gaussian_psi_3_1}
\begin{aligned}
&\int_\R e^{-2i\lambda x_1}\int_{M_0} (h(\nabla_g \psi) (v_s)  -(\nabla_g \psi)_1 v_s) \overline{w_s} c(x_1,x')^{\frac{n}{2}} dV_{g_0}dx_1\\
&= \sum_{l:\gamma(t_l)=p_j} \int_\R e^{-2i\lambda x_1}\int_{M_0} (h(\nabla_g \psi) (v_s^{(l)})  -(\nabla_g \psi)_1 v_s^{(l)}) \overline{w_s^{(l)}} c(x_1,x')^{\frac{n}{2}} dV_{g_0}dx_1+\\
&\sum_{l\ne l':\gamma(t_l)=\gamma(t_{l'})=p_j}\int_\R e^{-2i\lambda x_1}\int_{M_0} (h(\nabla_g \psi) (v_s^{(l)})  -(\nabla_g \psi)_1 v_s^{(l)}) \overline{w_s^{(l')}} c(x_1,x')^{\frac{n}{2}} dV_{g_0}dx_1.
\end{aligned}
\end{equation}
Now when $l\ne l'$, as in \eqref{eq_gauss_33} and \eqref{eq_200_59_0}, by non-stationary phase argument we see that 
\[
\lim_{h\to 0}\int_{M_0} (h(\nabla_g \psi) (v_s^{(l)})  -(\nabla_g \psi)_1 v_s^{(l)}) \overline{w_s^{(l')}} c(x_1,x')^{\frac{n}{2}} dV_{g_0}=0,
\]
uniformly in $x_1$, and therefore, the limit $h\to 0$ of the second sum in \eqref{eq_prop_gaussian_psi_3_1} is equal to $0$. 
Hence, 
\begin{align*}
&\lim_{h\to 0}\int_\R e^{-2i\lambda x_1}\int_{M_0} (h(\nabla_g \psi) (v_s)  -(\nabla_g \psi)_1 v_s) \overline{w_s} c(x_1,x')^{\frac{n}{2}} dV_{g_0}dx_1\\
&= \sum_{l:\gamma(t_l)=p_j}\int_\R  \int_{I_l} e^{-2i\lambda (x_1-it)}  \psi(x_1,t,0)c(x_1,t,0)
dtdx_1,
\end{align*}
showing \eqref{eq_prop_gaussian_psi} when  $\supp(\psi (x_1,\cdot))\subset V_j$. 
\end{proof}

\section{Construction of complex geometric optics solutions based on Gaussian beams quasimodes}
\label{sec_CGO}

Let $(M,g)$ be a conformally transversally anisotropic manifold so that $(M,g)\subset\subset (\R\times M_0^{\text{int}}, c(e\oplus g_0))$.  Let $X,Y\in L^\infty(M,TM)$ be complex vector fields, and let $q\in L^\infty(M,\C)$. Consider the following operator, 
\begin{equation}
\label{eq_100_1}
P_{X,Y,q}=(-\Delta_g)^2+X+\div(Y)+q. 
\end{equation}
Note that the operator $P_{X,Y,q}$ comprises both the operator $L_{X,q}$ as well as  its formal adjoint $L^*_{X,q}=(-\Delta_g)^2-\overline{X}-\div(\overline{X})+\overline{q}$.
Here $\div(Y)\in H^{-1}(M^{\text{int}})$ is given by 
\begin{equation}
\label{eq_100_1_1}
\langle \div(Y), \varphi\rangle_{M^{\text{int}}}:=-\int Y(\varphi)dV, \quad \varphi\in C^\infty_0(M^{\text{int}}),
\end{equation}
where $\langle \cdot,\cdot\rangle_{M^{\text{int}}}$ is a distributional duality on $M^{\text{int}}$. We shall also view $\div(Y)$ as multiplication operator,
\begin{equation}
\label{eq_100_1_2}
\div(Y): C_0^\infty(M^{\text{int}})\to H^{-1}(M^{\text{int}}).
\end{equation}
Therefore, it follows from \eqref{eq_100_1} that 
\[
P_{X,Y,q}:C_0^\infty(M^{\text{int}})\to H^{-1}(M^{\text{int}}).
\]
In this section, we will construct complex geometric optics solutions to the equation $P_{X,Y,q}u=0$ in $M$ based on the Gaussian beams quasimodes for the conjugated biharmonic operator, constructed in Section \ref{sec_Gaussian_beams}.

Assume, as we may,  that $(M,g)$ is embedded in a compact smooth manifold $(N,g)$ without boundary of the same dimension, and let $U$ be open in $N$ such that $M\subset U$.  Our starting point is the following Carleman estimates for $-h^2\Delta_g$ with a gain of two derivatives, established in \cite{KK_2018}, see also \cite{DKSaloU_2009} and \cite{Salo_Tzou_2009}.

\begin{prop}
Let $\phi$ be a limiting Carleman weight for $-h^2\Delta_g$ on  $U$. Then for all $0<h\ll 1$ and $t\in \R$, we have
\begin{equation}
\label{eq_Car_for_laplace}
h\|u\|_{H^{t+2}_{\emph{\text{scl}}}(N)}\le C\|e^{\frac{\phi }{h}}(-h^2\Delta_g)e^{-\frac{\phi}{h}}  u\|_{H^{t}_{\emph{\text{scl}}}(N)}, \quad C>0,
\end{equation}
for all $u\in C_0^\infty(M^\text{int})$. 
\end{prop} 
Here $H^t(N)$, $t\in\R$, is the standard  Sobolev space, equipped with the natural semiclassical norm,
\[
\|u\|_{H^t_{\text{scl}}(N)}=\|(1-h^2\Delta_g)^{\frac{t}{2}} u\|_{L^2(N)}. 
\]
Iterating  \eqref{eq_Car_for_laplace}, we get the following Carleman estimates for $(-h^2\Delta_g)^2$, for $0<h\ll 1$ and $t\in \R$, 
\begin{equation}
\label{eq_Car_for_biharm}
h^2\|u\|_{H^{t+4}_{\emph{\text{scl}}}(N)}\le C\|e^{\frac{\phi }{h}}(-h^2\Delta_g)^2e^{-\frac{\phi}{h}}  u\|_{H^{t}_{\emph{\text{scl}}}(N)}, \quad C>0,
\end{equation}
for all $u\in C_0^\infty(M^\text{int})$.

To construct complex geometric optics solutions for $P_{X,Y,q}u=0$, we shall need the following Carleman estimates for the operator $P_{X,Y,q}$.
\begin{prop}
Let $\phi$ be a limiting Carleman weight for $-h^2\Delta_g$ on  $U$. Then for all $0<h\ll 1$, we have
\begin{equation}
\label{eq_100_2}
h^2\|u\|_{H^1_{\text{scl}}(N)}\le C\|e^{\frac{\phi }{h}}(h^4P_{X,Y,q})e^{-\frac{\phi}{h}}  u \|_{H^{-3}_{\text{scl}}(N)}, \quad C>0,
\end{equation}
for all $u\in C^\infty_0(M^{\text{int}})$.
\end{prop}

\begin{proof}
First letting $t=-3$ in \eqref{eq_Car_for_biharm}, we get for all $0<h\ll 1$, 
\begin{equation}
\label{eq_100_3}
h^2\|u\|_{H^{1}_{\emph{\text{scl}}}(N)}\le C\|e^{\frac{\phi }{h}}(-h^2\Delta_g)^2e^{-\frac{\phi}{h}}  u\|_{H^{-3}_{\emph{\text{scl}}}(N)}, 
\end{equation}
for all $u\in C^\infty_0(M^{\text{int}})$. We also have 
\begin{equation}
\label{eq_100_4}
\|e^{\frac{\phi }{h}} h^4 X(e^{-\frac{\phi}{h}}  u)\|_{H^{-3}_{\emph{\text{scl}}}(N)}\le \|h^4X(u)-h^3X(\phi)u\|_{L^2(N)}=\mathcal{O}(h^3)\|u\|_{H^{1}_{\emph{\text{scl}}}(N)}.
\end{equation}
In order to estimate $\|h^4\div(Y)u \|_{H^{-3}_{\text{scl}}(N)}$,  we shall use the following characterization of the semiclassical norm in the Sobolev space $H^{-3}(N)$, 
\[
\|v\|_{H^{-3}_{\text{scl}}(N)}=\sup_{0\ne \psi\in C^\infty(N)}\frac{|\langle v,\psi\rangle_N|}{\| \psi\|_{H^3_{\text{scl}}(N)}}. 
\]
Using \eqref{eq_100_1_1}, for $0\ne \psi\in C^\infty(N)$, we get 
\[
|\langle h^4e^{\frac{\phi}{h}}\div(Y)e^{-\frac{\phi}{h}} u,\psi\rangle_N|\le \int_N h^4|Y(u\psi)|dV\le \mathcal{O}(h^3)\|u\|_{H^{1}_{\text{scl}}(N)}\|\psi\|_{H^{3}_{\text{scl}}(N)},
\]
and therefore, 
\begin{equation}
\label{eq_100_5}
\|h^4\div(Y)u \|_{H^{-3}_{\text{scl}}(N)}\le \mathcal{O}(h^3)\|u\|_{H^{1}_{\text{scl}}(N)}.
\end{equation}
Finally, we have
\begin{equation}
\label{eq_100_6}
\|h^4qu\|_{H^{-3}_{\text{scl}}(N)}\le \mathcal{O}(h^4)\|u\|_{H^{1}_{\text{scl}}(N)}.
\end{equation}
Combining \eqref{eq_100_3}, \eqref{eq_100_4}, \eqref{eq_100_5}, and \eqref{eq_100_6}, we obtain \eqref{eq_100_2}, for all $0<h\ll 1$ and $u\in C^\infty_0(M^{\text{int}})$.
\end{proof}

Note that the formal $L^2$ adjoint of $P_{X,Y,q}$ is given by $P_{-\overline{X}, -\overline{X}+\overline{Y},\overline{q}}$. Using the fact that  if $\phi$ is a limiting Carleman weight then so is $-\phi$, we obtain the following solvability result, see  \cite{DKSaloU_2009} and \cite{Krup_Uhlmann_magnet} for the details. 
\begin{prop}
\label{prop_solvability}
Let $X, Y\in L^\infty(M,TM)$ be complex vector fields, and  let $q\in L^\infty(M, \C)$. Let  $\phi$ be a limiting Carleman weight for $-h^2\Delta_g$ on $(U,g)$. 
If  $h>0$ is small enough, then for any $v\in H^{-1}(M^{\text{int}})$, there is a solution $u\in H^3(M^{\text{int}})$ of the equation 
\[
e^{\frac{\phi}{h}}(h^4P_{X,Y,q})e^{-\frac{\phi}{h}}  u=v \quad \text{in}\quad M^{\text{int}},
\]
which satisfies 
\[
\|u\|_{H^3_{\emph{\text{scl}}}(M^{\text{int}})}\le \frac{C}{h^2} \|v\|_{H^{-1}_{\emph{\text{scl}}}(M^{\text{int}})}.
\]
\end{prop}

Let 
\[
s=\mu+i\lambda, \quad 1\le \mu=\frac{1}{h}, \quad \lambda\in \R, \quad \lambda\quad\text{fixed}. 
\]
We shall construct complex geometric optics solutions to the equation 
\begin{equation}
\label{eq_100_7}
P_{X,Y,q}u=0\quad \text{in}\quad M^{\text{int}}
\end{equation}
of the form
\begin{equation}
\label{eq_100_8}
u=e^{-sx_1}(v_s+r_s),
\end{equation}
where $v_s$ is a Gaussian beam quasimode for $(-h^2\Delta_g)^2$, constructed in Proposition \ref{prop_Gaussian_beams}. 
Thus, $u$ is a solution to \eqref{eq_100_7} provided that 
\begin{equation}
\label{eq_100_9}
\begin{aligned}
e^{sx_1}h^4P_{X,Y,q}e^{-sx_1}r_s=&-e^{sx_1}h^4P_{X,Y,q}e^{-sx_1}v_s=-e^{sx_1}(-h^2\Delta_g)^2e^{-sx_1}v_s
\\
&- e^{sx_1}h^4X(e^{-sx_1}v_s)-e^{sx_1}h^4\div(Y)(e^{-sx_1}v_s)-h^4qv_s=:F.
\end{aligned}
\end{equation}
Let us estimates the terms in the right hand side of \eqref{eq_100_9} in $H^{-1}_{\text{scl}}(M^{\text{int}})$. First, it follows from \eqref{eq_prop_gaussian_1} that 
\begin{equation}
\label{eq_100_10_0}
\| e^{sx_1}(-h^2\Delta_g)^2 e^{-sx_1}v_s\|_{H^{-1}_{\text{scl}}(M^{\text{int}})}\le\| e^{sx_1}(-h^2\Delta_g)^2 e^{-sx_1}v_s\|_{L^2(M)}=\mathcal{O}(h^{5/2}),
\end{equation}
and
\begin{equation}
\label{eq_100_10}
\| e^{sx_1}h^4X(e^{-sx_1}v_s)\|_{H^{-1}_{\text{scl}}(M^{\text{int}})}\le \| h^4 X(v_s)-h^4s X(x_1)v_s\|_{L^2(M)}=\mathcal{O}(h^3).
\end{equation}
Letting $0\ne\rho\in C^\infty_0(M^{\text{int}})$ and using \eqref{eq_100_1_1}, we obtain that 
\begin{align*}
|\langle e^{sx_1}h^4\div(Y)&(e^{-sx_1}v_s),\rho\rangle_{M^{\text{int}}}|\le  h^4\int |Y(v_s\rho)|dV\\
&=\mathcal{O}(h^3)\|v_s\|_{H^1_{\text{scl}}(M^{\text{int}})}\|\rho\|_{H^1_{\text{scl}}(M^{\text{int}})}
=\mathcal{O}(h^3)\|\rho\|_{H^1_{\text{scl}}(M^{\text{int}})},
\end{align*}
and therefore, 
\begin{equation}
\label{eq_100_11}
\| e^{sx_1}h^4\div(Y)(e^{-sx_1}v_s)\|_{H^{-1}_{\text{scl}}(M^{\text{int}})}=\mathcal{O}(h^3).
\end{equation}
We also have 
\begin{equation}
\label{eq_100_12}
\|h^4qv_s\|_{H^{-1}_{\text{scl}}(M^{\text{int}})}=\mathcal{O}(h^4). 
\end{equation}
Using \eqref{eq_100_10_0}, \eqref{eq_100_10}, \eqref{eq_100_11}, \eqref{eq_100_12}, we get from \eqref{eq_100_9}  that $\|F\|_{H^{-1}_{\text{scl}}(M^{\text{int}})}=\mathcal{O}(h^{5/2})$. An application of Proposition \ref{prop_solvability} to  \eqref{eq_100_9} gives that for all $h>0$ small enough,  there exists  $r_s\in H^3(M^{\text{int}})$ such that 
$\|r_s\|_{H^3_{\emph{\text{scl}}}(M^{\text{int}})}=\mathcal{O}(h^{1/2})$. To summarize, we prove the following result. 
\begin{prop}
\label{prop_CGO}
Let $X, Y\in L^\infty(M,TM)$ be complex vector fields, and  let $q\in L^\infty(M, \C)$. Let $s=\frac{1}{h}+i\lambda$ with $\lambda\in \R$ being fixed. For all $h>0$ small enough, there is a solution $u_1\in H^3(M^{\text{int}})$ of $P_{X,Y,q}u_1=0$ in $M^{\text{int}}$ having the form
\[
u_1=e^{-sx_1}(v_s+r_1),
\]
where $v_s\in C^\infty(M)$ is the Gaussian beam quasimode given in Proposition \ref{prop_Gaussian_beams} and $r_1\in H^3(M^{\text{int}})$ such that 
$\|r_1\|_{H^3_{\emph{\text{scl}}}(M^{\text{int}})}=\mathcal{O}(h^{1/2})$ as $h\to 0$. 

Similarly, for all $h>0$ small enough, there is a solution $u_2\in H^3(M^{\text{int}})$ of $P_{X,Y,q}u_2=0$ in $M^{\text{int}}$ having the form
\[
u_2=e^{sx_1}(w_s+r_2),
\]
where $w_s\in C^\infty(M)$ is the Gaussian beam quasimode given in Proposition \ref{prop_Gaussian_beams} and $r_2\in H^3(M^{\text{int}})$ such that 
$\|r_2\|_{H^3_{\emph{\text{scl}}}(M^{\text{int}})}=\mathcal{O}(h^{1/2})$ as $h\to 0$. 
\end{prop}

\section{Proof of Theorem \ref{thm_main}}
\label{sec_proof_of_main_result}

Our starting point is the following integral identity. 

\begin{prop}
\label{prop_integral_identity}
Let $X^{(1)},X^{(2)} \in C(M,TM)$ with complex valued coefficients, and $q^{(1)},q^{(2)} \in C(M,\C)$. If $\mathcal{C}_{X^{(1)},q^{(1)}} =\mathcal{C}_{X^{(2)},q^{(2)}} $, then 
\begin{equation}
\label{eq_integral_identity}
\int_M \big((X^{(1)}-X^{(2)})(u_1)\overline{u_2}+(q^{(1)}-q^{(2)})u_1\overline{u_2}\big) dV_g = 0, 
\end{equation}
for $u_1,u_2 \in H^3(M^\text{int})$ satisfying 
\begin{equation}
\label{eq_int_identity_equations}
L_{X^{(1)},q^{(1)}} u_1 = 0 \quad \text{and}\quad L_{-\overline{X^{(2)}},-\div (\overline{X^{(2)}})+\overline{q^{(2)}}} u_2 = 0.
\end{equation}
\end{prop}

\begin{proof}
First, using that $\overline{u_2}$ solves the equation
\begin{equation}
\label{eq_int_identity_equations_1}
L_{-X^{(2)},-\div (X^{(2)})+q^{(2)}} \overline{u_2} = 0,
\end{equation}
similar to \eqref{eq_int_2}, we define the boundary trace $\p_\nu(\Delta_g \overline{u_2})\in H^{-1/2}(\p M)$ as follows. Letting $\varphi\in H^{1/2}(\p M)$ and letting $v\in H^1(M^{\text{int}})$ be a continuous extension of $\varphi$, we set
\begin{equation}
\label{eq_int_id_100_1}
\begin{aligned}
\langle \p_\nu (-\Delta_g \overline{u_2}), \varphi \rangle_{H^{-1/2}(\p M)\times H^{1/2}(\p M)}=-\int_{\p M} (X^{(2)}\cdot\nu)\overline{u_2}vdS_g\\
+\int_M \big( \langle \nabla_g (-\Delta_g \overline{u_2}), \nabla_g v\rangle_g +\overline{u_2} X^{(2)}(v)+q^{(2)}\overline{u_2}v\big)dV_g.
\end{aligned}
\end{equation}
It follows from \eqref{eq_int_identity_equations_1} that the definition of the trace $\p_\nu(\Delta_g \overline{u_2})$ is independent of the choice of extension $v$ of $\varphi$. 

As $\mathcal{C}_{X^{(1)},q^{(1)}} =\mathcal{C}_{X^{(2)},q^{(2)}} $, there exists $v_2\in H^3(M^{\text{int}})$ such that 
\begin{equation}
\label{eq_int_id_100_2_-1}
L_{X^{(2)},q^{(2)}}v_2=0\quad \text{in}\quad M,
\end{equation}
and 
\begin{equation}
\label{eq_int_id_100_2}
\begin{aligned}
u_1|_{\p M}=v_2|_{\p M},\quad (\Delta_g u_1)|_{\p M}=(\Delta_g v_2)|_{\p M},\quad \p_\nu u_1|_{\p M}=\p_\nu v_2|_{\p M}, \\\p_\nu(\Delta_g u_1)|_{\p M}=\p_\nu(\Delta_g v_2)|_{\p M}.
\end{aligned}
\end{equation}
It follows from \eqref{eq_int_id_100_2} in particular that 
\begin{equation}
\label{eq_int_id_100_3}
\langle \p_\nu(\Delta_g u_1), \overline{u_2} \rangle_{H^{-1/2}(\p M)\times H^{1/2}(\p M)}=\langle\p_\nu(\Delta_g v_2), \overline{u_2} \rangle_{H^{-1/2}(\p M)\times H^{1/2}(\p M)}.
\end{equation}
Using that $v_2$ solves \eqref{eq_int_id_100_2_-1} and \eqref{eq_int_2}, we get 
\begin{equation}
\label{eq_int_id_100_4}
\begin{aligned}
\langle\p_\nu(-\Delta_g v_2), \overline{u_2} &\rangle_{H^{-1/2}(\p M)\times H^{1/2}(\p M)}\\
&=\int_M \big( \langle \nabla_g (-\Delta_g v_2), \nabla_g \overline{u_2}\rangle_g +X^{(2)}(v_2)\overline{u_2}+q^{(2)}v_2\overline{u_2}\big)dV_g.
\end{aligned}
\end{equation}
Using \eqref{eq_int_id_100_1} and integration by parts, we obtain that 
\begin{equation}
\label{eq_int_id_100_5}
\begin{aligned}
\langle \p_\nu (-\Delta_g \overline{u_2}), &v_2\rangle_{H^{-1/2}(\p M)\times H^{1/2}(\p M)}=-\int_{\p M} (X^{(2)}\cdot\nu)\overline{u_2}v_2dS_g\\
&+\int_M \big( \langle \nabla_g  \overline{u_2}, \nabla_g  (-\Delta_g)v_2\rangle_g +\overline{u_2} X^{(2)}(v_2)+q^{(2)}\overline{u_2}v_2\big)dV_g\\
&+\int_{\p M}(\p_\nu \overline{u_2})\Delta_g v_2 dS_g-\int_{\p M}(\Delta_g \overline{u_2})\p_\nu v_2dS_g.
\end{aligned}
\end{equation}
Combining \eqref{eq_int_id_100_4} and \eqref{eq_int_id_100_5}, using \eqref{eq_int_id_100_2},  we obtain that 
\begin{equation}
\label{eq_int_id_100_6}
\begin{aligned}
\langle\p_\nu(-\Delta_g v_2), &\overline{u_2} \rangle_{H^{-1/2}(\p M)\times H^{1/2}(\p M)}=\langle \p_\nu (-\Delta_g \overline{u_2}), v_2\rangle_{H^{-1/2}(\p M)\times H^{1/2}(\p M)}\\
&+\int_{\p M} (X^{(2)}\cdot\nu)\overline{u_2}v_2dS_g-\int_{\p M}(\p_\nu \overline{u_2})\Delta_g v_2 dS_g+\int_{\p M}(\Delta_g \overline{u_2})\p_\nu v_2dS_g\\
&=\langle \p_\nu (-\Delta_g \overline{u_2}), u_1\rangle_{H^{-1/2}(\p M)\times H^{1/2}(\p M)}
+\int_{\p M} (X^{(2)}\cdot\nu)\overline{u_2}u_1dS_g\\
&-\int_{\p M}(\p_\nu \overline{u_2})\Delta_g u_1 dS_g+\int_{\p M}(\Delta_g \overline{u_2})\p_\nu u_1dS_g\\
&=\int_M \big( \langle \nabla_g  \overline{u_2}, \nabla_g  (-\Delta_g)u_1\rangle_g +\overline{u_2} X^{(2)}(u_1)+q^{(2)}\overline{u_2}u_1\big)dV_g.
\end{aligned}
\end{equation}
On the other hand, using the equation \eqref{eq_int_identity_equations} for $u_1$ and \eqref{eq_int_2}, we get 
\begin{equation}
\label{eq_int_id_100_7}
\begin{aligned}
\langle \p_\nu(-\Delta_g u_1), &\overline{u_2} \rangle_{H^{-1/2}(\p M)\times H^{1/2}(\p M)}\\
&=\int_M \big( \langle \nabla_g  (-\Delta_g)u_1, \nabla_g  \overline{u_2}\rangle_g +X^{(1)}(u_1)\overline{u_2} +q^{(1)}u_1\overline{u_2}\big)dV_g.
\end{aligned}
\end{equation}
The claim follows from \eqref{eq_int_id_100_3}, \eqref{eq_int_id_100_6} and \eqref{eq_int_id_100_7}. 
\end{proof}

Now by Proposition \ref{prop_CGO}, for $h>0$ small enough, there are $u_1, u_2\in H^3(M^{\text{int}})$ solutions to  $L_{X^{(1)},q^{(1)}} u_1 = 0$ and $L_{-\overline{X^{(2)}},-\div (\overline{X^{(2)}})+\overline{q^{(2)}}} u_2 = 0$ in $M^{\text{int}}$,  of the form, 
\begin{equation}
\label{eq_300_1}
u_1=e^{-sx_1}(v_s+r_1), \quad u_2=e^{sx_1}(w_s+r_2),
\end{equation}
where $v_s, w_s\in C^\infty(M)$ are the Gaussian beam quasimode given in Proposition \ref{prop_Gaussian_beams} and 
\begin{equation}
\label{eq_300_2}
\|r_1\|_{H^1_{\emph{\text{scl}}}(M^{\text{int}})}=\mathcal{O}(h^{1/2}), \quad \|r_2\|_{H^1_{\emph{\text{scl}}}(M^{\text{int}})}=\mathcal{O}(h^{1/2}), 
\end{equation}
 as $h\to 0$. 
 
Let us denote $X = X^{(1)}-X^{(2)}$ and $q= q^{(1)}- q^{(2)}$. By the boundary determination of Proposition \ref{boundary determination_X}, we may extend $X$ by zero to the complement of $M$ in $\R\times M_0$ so that the extension $X\in C(\R\times M_0,T(\R\times M_0))$.

\textit{Step 1. Proving that there exists $\psi\in C^1(\R\times M_0)$ with compact support in $x_1$ such that $\psi(x_1,\cdot)|_{\p M_0}=0$ and $\nabla_g \psi=X$.}

In this step, we shall work with solutions $u_1$ and $u_2$ given by \eqref{eq_300_1} with $v_s$ and $w_s$ being the Gaussian beam quasimode for which Proposition \ref{prop_computations_gaussian_1} holds. In particular, here $v_s$ has an amplitude of the first type. Next, we would like to substitute $u_1$ and $u_2$ into the integral identity \eqref{eq_integral_identity}, multiply it by $h$, and let $h\to 0$. To that end, first using \eqref{eq_300_2}, \eqref{eq_prop_gaussian_1}, and \eqref{eq_prop_gaussian_2}, we get
\begin{equation}
\label{eq_300_3}
\bigg|h \int_M q u_1\overline{u_2} dV_g \bigg|=\bigg|h \int_M q e^{-2i\lambda x_1}(v_s+r_1)(\overline{w_s}+\overline{r_2}) dV_g \bigg|=\mathcal{O}(h).
\end{equation}

Writing $x=(x_1,x')$, $x'\in M_0$, and $X=X_1\p_{x_1}+\tilde X\cdot\p_{x'}$, we obtain that 
\begin{equation}
\label{eq_300_4}
h\int_M X(u_1)\overline{u_2} dV_g=I_1+I_2+I_3+I_4,  
\end{equation}
where 
\begin{equation}
\label{eq_300_5}
I_1=h\int_M e^{-2i\lambda x_1} X(v_s)\overline{w_s}dV_g-\int_M X_1(x_1,x')e^{-2i\lambda x_1}v_s\overline{w_s}dV_g,
\end{equation}
\begin{equation}
\label{eq_300_6}
I_2=-hi\lambda\int_M X_1(x_1,x')e^{-2i\lambda x_1}(v_s+r_1)(\overline{w_s}+\overline{r_2})dV_g,
\end{equation}
\begin{equation}
\label{eq_300_7}
I_3=-\int_M X_1(x_1,x')e^{-2i\lambda x_1}(v_s \overline{r_2} + \overline{w_s} r_1+ r_1\overline{r_2})dV_g,
\end{equation}
\begin{equation}
\label{eq_300_8}
I_4=h\int_M e^{-2i\lambda x_1} (X(v_s) \overline{r_2}+ X(r_1) \overline{w_s}+ X(r_1)\overline{r_2})dV_g.
\end{equation}
Using \eqref{eq_300_2}, \eqref{eq_prop_gaussian_1}, and \eqref{eq_prop_gaussian_2}, we get
\begin{equation}
\label{eq_300_9}
|I_2|=\mathcal{O}(h), \quad |I_3|=\mathcal{O}(h^{1/2}), \quad |I_4|=\mathcal{O}(h^{1/2}).
\end{equation}
It follows from \eqref{eq_integral_identity} with the help of \eqref{eq_300_3}, \eqref{eq_300_4}, and \eqref{eq_300_9} that 
\begin{equation}
\label{eq_300_10}
\lim_{h\to 0}I_1=0.
\end{equation}
Using that $X=0$ outside of $M$, $dV_g=c^{\frac{n}{2}}dx_1dV_{g_0}$,  Fubini's theorem, and Proposition \ref{prop_computations_gaussian_1}, we obtain from \eqref{eq_300_10} that 
\begin{equation}
\label{eq_300_11}
\begin{aligned}
0=&\lim_{h\to 0} h\int_{\R} e^{-2i\lambda x_1} \int_{M_0} X(v_s)\overline{w_s}c(x_1,x')^{\frac{n}{2}}dV_{g_0}dx_1\\
&-\lim_{h\to 0}\int_\R e^{-2i\lambda x_1}\int_{M_0}   X_1(x_1,x')v_s\overline{w_s}c(x_1,x')^{\frac{n}{2}}dV_{g_0}dx_1\\
=&-\int_\R e^{-2i\lambda x_1}\int_0^L\big( X_1(x_1,\gamma(t))-iX_t(x_1,\gamma(t))\big)c(x_1,\gamma(t))e^{-2\lambda t}dtdx_1.
\end{aligned}
\end{equation}

Now the Riemmanian metric $g$ on $M$ induces a natural isomorphism between the tangent and cotangent bundles given by 
\begin{equation}
\label{eq_300_12_0}
TM\to T^*M, \quad (x,X)\mapsto (x,X^b),
\end{equation}
where $X^b(Y)=\langle X,Y\rangle$. In local coordinates, $X^b=\sum_{j,k=1}^n g_{jk}X_jdx_k$, and using that $g=c(e\oplus g_0)$, and \eqref{eq_gauss_3_metric}, we get 
\[
X_1^b(x_1,\gamma(t))=c(x_1,\gamma(t))X_1(x_1,\gamma(t)), \quad X_t^b(x_1,\gamma(t))=c(x_1,\gamma(t))X_t(x_1,\gamma(t)).
\]
Hence, it follows from \eqref{eq_300_11}, replacing $2\lambda$ by $\lambda$ that 
\begin{equation}
\label{eq_300_12}
\int_\R\int_0^L e^{-i\lambda x_1-\lambda t}(X_1^b (x_1,\gamma(t))-iX^b_t(x_1,\gamma(t)))dtdx_1=0.
\end{equation}
Letting 
\begin{equation}
\label{eq_300_13}
\begin{aligned}
&f(\lambda,x')=\int_\R e^{-i\lambda x_1} X_1^b (x_1,x')dx_1,\quad x'\in M_0, \\
&\alpha(\lambda,x')=\sum_{j=2}^n \bigg(\int_\R e^{-i\lambda x_1 } X^b_j(x_1,x') \bigg)dx_j,
\end{aligned}
\end{equation}
we have $f(\lambda,\cdot)\in C(M_0)$, $\alpha(\lambda, \cdot)\in C(M_0, T^*M)$, and \eqref{eq_300_12} implies that 
\begin{equation}
\label{eq_300_14}
\int_0^L[f(\lambda, \gamma(t))-i\alpha(\lambda,\dot{\gamma}(t))]e^{-\lambda t}dt=0, 
\end{equation}
along any unit speed non-tangential geodesic $\gamma:[0,L]\to M_0$ on $M_0$ and any $\lambda\in \R$. Arguing as in  \cite[Section 7]{KK_2018}, \cite{Cekic}, using the injectivity of the geodesic X-ray transform on functions and $1$-forms, we conclude from \eqref{eq_300_14} that there exist $p_l\in C^1(M_0)$, $p_l|_{\p M_0}=0$, such that 
\begin{equation}
\label{eq_300_15}
\p_\lambda^l f(0,x')+lp_{l-1}(x')=0, \quad \p_\lambda^l\alpha(0,x')=idp_l(x'),\quad l=0,1,2,\dots. 
\end{equation}
To proceed we shall follow \cite[Section 5]{LO_2019}, and let 
\begin{equation}
\label{eq_300_16}
\psi(x_1,x')=\int_{-a}^{x_1} X_1^b(y_1,x')dy_1,
\end{equation}
where $\supp(X^b(\cdot,x'))\subset (-a,a)$.  It follows from \eqref{eq_300_15}, \eqref{eq_300_13} that 
\[
0=f(0,x')=\int_\R X_1^b(y_1,x')dy_1,
\]
and therefore, $\psi$ has compact support in $x_1$. Thus, the Fourier transform of $\psi$ with respect to $x_1$, which we denote by $\hat \psi(\lambda, x')$ is real analytic with respect to $\lambda$, and therefore, we have 
\begin{equation}
\label{eq_300_17}
\hat \psi(\lambda,x')=\sum_{k=0}^\infty \frac{\psi_k(x')}{k!}\lambda^k,
\end{equation}
where $\psi_k(x')=(\p_\lambda^k\hat \psi)(0,x')$.  It follows from \eqref{eq_300_16} that 
\begin{equation}
\label{eq_300_18}
\p_{x_1}\psi(x_1,x')=X_1^b(x_1,x'),
\end{equation}
and therefore, taking the Fourier transform with respect to $x_1$, and using \eqref{eq_300_13}
\begin{equation}
\label{eq_300_19}
i\lambda \psi(\lambda,x')=f(\lambda,x').
\end{equation}
Differentiating \eqref{eq_300_19} $(l+1)$-times in $\lambda$, letting $\lambda=0$, and using \eqref{eq_300_15},  we get 
\begin{equation}
\label{eq_300_20}
\p_\lambda^l\hat \psi(0,x')=ip_l(x'), \quad l=0,1,2,\dots.
\end{equation}
Substituting \eqref{eq_300_20} into \eqref{eq_300_17}, we obtain that 
\[
\hat \psi(\lambda,x')=\sum_{k=0}^\infty \frac{ip_l(x')}{k!}\lambda^k,
\]
and taking the differential in $x'$ in the sense of distributions, and using \eqref{eq_300_15}, \eqref{eq_300_13}, we see that 
\begin{equation}
\label{eq_300_21}
d_{x'}\hat \psi(\lambda,x')=\sum_{k=0}^\infty \frac{i d p_l(x')}{k!}\lambda^k=\sum_{k=0}^\infty \frac{\p_{\lambda}^k\alpha(0,x')}{k!}\lambda^k=\alpha(\lambda,x')=\sum_{j=2}^n \hat X_j^b(\lambda,x')dx_j.
\end{equation}
Taking the inverse Fourier transform $\lambda\mapsto x_1$ in \eqref{eq_300_21}, we get 
\begin{equation}
\label{eq_300_22}
d_{x'} \psi(x_1,x')=\sum_{j=2}^n  X_j^b(x_1,x')dx_j.
\end{equation}
We also have from \eqref{eq_300_18} that 
\begin{equation}
\label{eq_300_23}
d_{x_1} \psi(x_1,x')=X_1^b(x_1,x')dx_1.
\end{equation}
It follows from \eqref{eq_300_22} and \eqref{eq_300_23} that 
\begin{equation}
\label{eq_300_24}
d\psi=X^b.
\end{equation}
Using the inverse of \eqref{eq_300_12_0}, we see from \eqref{eq_300_24} that 
\begin{equation}
\label{eq_300_25}
\nabla_g \psi=X.
\end{equation}
Recall that $\psi\in C(\R\times M_0)$ with compact support in $x_1$ and $\psi(x_1,\cdot)|_{\p M_0}=0$. It follows from \eqref{eq_300_25} that $\psi\in C^1(\R\times M_0)$.

\textit{Step 2. Showing that $X=0$.}

Returning to \eqref{eq_integral_identity} and using \eqref{eq_300_25}, we get
\begin{equation}
\label{eq_300_26}
\int_M \big((\nabla_g \psi)(u_1)\overline{u_2}+q u_1\overline{u_2}\big) dV_g = 0, 
\end{equation}
for $u_1,u_2 \in H^3_{scl}(M^\text{int})$ satisfying $L_{X^{(1)},q^{(1)}} u_1 = 0$ and $L_{-\overline{X^{(2)}},-\div (\overline{X^{(2)}})+\overline{q^{(2)}}} u_2 = 0$.
Let now $u_1$ and $u_2$ be given by \eqref{eq_300_1} with $v_s$ and $w_s$ being the Gaussian beam quasimode for which Proposition \ref{prop_computations_gaussian_2} holds. In particular, here $v_s$ has an amplitude of the second type. We would like to substitute $u_1$ and $u_2$ into the integral identity \eqref{eq_300_26}, multiply it by $h$, and let $h\to 0$. Similar to \eqref{eq_300_10}, using \eqref{eq_300_3} and \eqref{eq_300_9}, we get  
\begin{equation}
\label{eq_300_27}
\lim_{h\to 0}h\int_M e^{-2i\lambda x_1} (\nabla_g\psi)(v_s)\overline{w_s}dV_g-\int_M (\nabla_g\psi)_1e^{-2i\lambda x_1}v_s\overline{w_s}dV_g=0.
\end{equation}
It follows from \eqref{eq_300_27} with the help of Proposition \ref{prop_computations_gaussian_2}, 
\begin{equation}
\label{eq_300_28}
\int_\R  \int_0^L e^{-2i\lambda (x_1-it)}  \psi(x_1,\gamma(t)) c(x_1,\gamma(t))
dtdx_1=0.
\end{equation}
Now \eqref{eq_300_28} can be written as 
\begin{equation}
\label{eq_300_29}
\int_\gamma \hat{\psi c}(2\lambda, \gamma(t))e^{-2\lambda t}dt=0,
\end{equation}
for any $\lambda\in\R$ and any non-tangential geodesic $\gamma$ in $M_0$, where
\[
\hat{\psi c}(2\lambda, x')=\int_{-\infty}^\infty e^{-2 i\lambda x_1}(\psi c)(x_1,x')dx_1.
\]
The equation \eqref{eq_300_29} says that the attenuated geodesic ray transform of $\hat{\psi c}$ with constant attenuation $-2\lambda$ vanishes along all non-tangential geodesics in $M_0$. Arguing as in \cite[Proof of Theorem 1.2]{DKuLS_2016} and using the injectivity of the geodesic $X$-ray transform on functions, we conclude that $\psi c=0$, and therefore, $\psi=0$ and hence, $X=0$. 

\textit{Step 3. Proving that $q=0$.}

Returning to \eqref{eq_integral_identity} and substituting $X^{(1)}=X^{(2)}$, we get 
\begin{equation}
\label{eq_300_30}
\int_M q u_1\overline{u_2} dV_g = 0, 
\end{equation}
for $u_1,u_2 \in H^3_{scl}(M^\text{int})$ satisfying $L_{X^{(1)},q^{(1)}} u_1 = 0$ and $L_{-\overline{X^{(2)}},-\div (\overline{X^{(2)}})+\overline{q^{(2)}}} u_2 = 0$.
Let now $u_1$ and $u_2$ be given by \eqref{eq_300_1} with $v_s$ and $w_s$ being the Gaussian beam quasimode for which Proposition \ref{prop_computations_gaussian_1} holds. In particular, here $v_s$ has an amplitude of the first type. Substituting $u_1$ and $u_2$ into \eqref{eq_300_30}, we obtain that 
\begin{equation}
\label{eq_300_31}
0=\int_M q u_1\overline{u_2} dV_g = I_1+I_2, 
\end{equation}
where 
\begin{align*}
I_1&=\int_M e^{-2i\lambda x_1} q v_s\overline{w_s} dV_g=\int_\R e^{-2i\lambda x_1} \int_{M_0} q v_s\overline{w_s}  c^{\frac{n}{2}}dV_{g_0}dx_1,\\
I_2&=\int_M e^{-2i\lambda x_1} q  (v_s\overline{r_2}+r_1\overline{w_s}+r_1\overline{r_2}) dV_g.
\end{align*}
Here in view of the assumption \eqref{eq_int_boundaty},  we extended $q$ by zero to the complement of $M$ in $\R\times M_0$ so that the extension $q\in C(\R\times M_0,\C)$. 

Using \eqref{eq_300_2}, \eqref{eq_prop_gaussian_1}, and \eqref{eq_prop_gaussian_2}, we see that 
\begin{equation}
\label{eq_300_32}
|I_2|=\mathcal{O}(h^{1/2}).
\end{equation}

Letting $h\to 0$, we obtain from \eqref{eq_300_31}, \eqref{eq_300_32} with the help of Proposition \ref{prop_computations_gaussian_1}  that 
\[
\int_\R e^{-2i\lambda x_1} \int_0^L e^{-2\lambda t}  (qc)(x_1,\gamma(t))dtdx_1=0.
\]
Arguing as in \cite[Proof of Theorem 1.2]{DKuLS_2016} and using the injectivity of the geodesic $X$-ray transform on functions, we conclude that $q c=0$, and therefore, $q=0$. This complete the proof of Theorem \ref{thm_main}.

\begin{appendix}
\section{Boundary determination of a first order perturbation of the biharmonic operator}

\label{sec_boundary_rec}

When proving Theorem \ref{thm_main}, an important step consists in determining the boundary values of the first order perturbation of the biharmonic operator. The purpose of this section is to carry out this step by adapting the method of \cite{Brown_Salo_2006},  \cite{KK_2018}.

\begin{prop}
\label{boundary determination_X}
Let $(M,g)$ be a CTA manifold of dimension $n\ge 3$. Let $X^{(1)},X^{(2)} \in C(M,TM)$ with complex vector fields and $q^{(1)},q^{(2)} \in L^\infty(M,\C)$. If $\mathcal{C}_{g,X^{(1)},q^{(1)}} =\mathcal{C}_{g,X^{(2)},q^{(2)}} $, then $X^{(1)}|_{\p M} = X^{(2)}|_{\p M}$.
\end{prop}

\begin{proof}
We shall follow \cite{Brown_Salo_2006}, \cite{KK_2018} closely. We shall construct some special solutions to the equations $L_{X^{(1)},q^{(1)}}u_1=0$ and $L_{-\overline{X^{(2)}},-\div(\overline{X^{(2)}})+\overline{q^{(2)}}}u_2=0$, whose boundary values have an oscillatory behavior while becoming increasingly concentrated near a given point on the boundary of $M$. Substituting these solutions into the integral identity \eqref{eq_integral_identity} will allow us to prove that $X^{(1)}|_{\p M} = X^{(2)}|_{\p M}$. 

In doing so, let $x_0\in \p M$ and let $(x_1,\dots, x_n)$ be the boundary normal coordinates centered at $x_0$ so that in these coordinates, $x_0 =0$, the boundary $\p M$ is given by $\{x_n=0\}$, and $M^{\text{int}}$ is given by $\{x_n > 0\}$. We shall assume, as we may, that 
\begin{equation}
\label{eq_Laplace_boundary-nc_2}
g^{\alpha \beta}(0)=\delta^{\alpha \beta}, \quad 1\le \alpha,\beta\le n-1,
\end{equation}
and therefore $T_0\p M=\R^{n-1}$, equipped with the Euclidean metric.  The unit tangent vector $\tau$ is then given by $\tau=(\tau',0)$ where $\tau'\in \R^{n-1}$, $|\tau'|=1$.   Associated to the tangent vector $\tau'$ is the covector $\xi'_\alpha=\sum_{\beta=1}^{n-1} g_{\alpha \beta}(0) \tau'_\beta=\tau'_\alpha\in T^*_{x_0}\p M$.

Let $\eta\in C^\infty_0(\R^n,\R)$ be a function such that $\supp(\eta)$ is in a small neighborhood of $0$, and 
\begin{equation}
\label{eq_int_eta_1}
\int_{\R^{n-1}}\eta(x',0)^2dx'=1.
\end{equation}
Following \cite{Brown_Salo_2006}, in the boundary normal coordinates,  we set 
\begin{equation}
\label{eq_9_1}
v_0(x)=\eta\bigg(\frac{x}{\lambda^{1/2}}\bigg)e^{\frac{i}{\lambda}(\tau'\cdot x'+ ix_n)}, \quad 0<\lambda\ll 1,
\end{equation}
so that  $v_0\in C^\infty(M)$ with $\supp(v_0)$ in  $\mathcal{O}(\lambda^{1/2})$ neighborhood of $x_0=0$. Here $\tau'$ is viewed as a covector. 

Let $v_1\in H^1_0(M^{\text{int}})$ be the solution to the following Dirichlet problem for the Laplacian,
\begin{equation}
\label{eq_9_Dirichlet_lapl}
\begin{aligned}
-\Delta_g v_1=&\Delta_g v_0, \quad\textrm{in}\quad M,\\
v_1|_{\p M}=&0.   
\end{aligned}
\end{equation} 
Let $\delta(x)$ be the distance from $x\in M$ to the boundary of $M$. As proved in the  \cite[Appendix]{KK_2018}, the following estimates hold:
\begin{equation}
\label{eq_9_3} 
\|v_0\|_{L^2(M)}\le \mathcal{O}(\lambda^{\frac{n-1}{4}+\frac{1}{2}}),
\end{equation}
\begin{equation}
\label{eq_9_4}
\|v_1\|_{L^2(M)}\le \mathcal{O}(\lambda^{\frac{n-1}{4}+\frac{1}{2}}),
\end{equation}
\begin{equation}
\label{eq_9_15}
\|d v_1\|_{L^2(M)}\le \mathcal{O}(\lambda^{\frac{n-1}{4}}),
\end{equation}
\begin{equation}
\label{eq_Brown_Salo_2_18_new_mnfld}
\| dv_0\|_{L^2(M)}\le \mathcal{O}(\lambda^{\frac{n-1}{4}-\frac{1}{2}}),
\end{equation}
\begin{equation}
\label{eq_Brown_Salo_2_20}
\|\delta d(v_0+v_1)\|_{L^2(M)}\le \mathcal{O}(\lambda^{\frac{n-1}{4}+\frac{1}{2}}),
\end{equation}
\begin{equation}
\label{eq_Brown_Salo_2_20_boundary}
\|v_0\|_{L^2(\p M)}\le \mathcal{O}(\lambda^{\frac{n-1}{4}}).
\end{equation}
We shall also need Hardy's inequality,
\begin{equation}
\label{eq_Hardy} 
\int_M |f(x)/\delta(x)|^2dV_g\le C\int_M |df(x)|^2dV_g,
\end{equation}
where $f\in H^1_0(M^{\text{int}})$, see \cite{Davies_2000}.

Next we would like to show the existence of a solution $u_1 \in H^3(M^{\text{int}})$ to the equation
\begin{equation}
\label{eq_A_3}
L_{X^{(1)},q^{(1)}} u_1 =0\quad  \text{in}\quad M,
\end{equation}
of the form 
\begin{equation}
\label{eq_A_4}
u_1=v_0+v_1+r_1,
\end{equation}
 with
 \begin{equation}
 \label{eq_A_r1}
 \|r_1\|_{H^{3}(M^{\text{int}})} \le \mathcal{O} (\lambda^{\frac{n-1}{4}+\frac{1}{2}}).
 \end{equation}
 
 To that end, plugging \eqref{eq_A_4} into \eqref{eq_A_3}, we obtain the following equation of $r_1$, 
 \begin{equation}
 \label{eq_A_5}
 L_{X^{(1)},q^{(1)}} r_1 = -((-\Delta_g)^2+X^{(1)}+q^{(1)}) (v_0+v_1) = -(X^{(1)}+q^{(1)})(v_0+v_1) \quad  \text{in}\quad M.
 \end{equation}
 Applying Proposition \ref{prop_solvability} with $h>0$ small but fixed, we conclude the existence of $r_1 \in H^3(M^{\text{int}})$ such that 
 \begin{equation}
 \label{eq_A_6}
 \|r_1\|_{H^3(M^{\text{int}})} \le \mathcal{O}(1)\|(X^{(1)}+q^{(1)})(v_0+v_1)  \|_{H^{-1}(M^{\text{int}})}. 
 \end{equation}
 Let us now bound the norm in the right hand side of \eqref{eq_A_6}. To that end,  letting $\psi \in C^\infty_0(M^{\text{int}})$ and using  \eqref{eq_Hardy}, \eqref{eq_Brown_Salo_2_20}, we get
 \begin{equation}
 \label{eq_A_7}
 \begin{aligned}
 |\langle X^{(1)}(v_0+v_1), \psi \rangle_{M^{\text{int}}}|  \le \mathcal{O}(1)\|X^{(1)}\|_{L^\infty(M)}  \|\delta d(v_0+v_1)\|_{L^2(M)}\|\psi\|_{H^1(M^{\text{int}})}\\
   \le \mathcal{O} (\lambda^{\frac{n-1}{4}+\frac{1}{2}}) \|\psi\|_{H^1(M^{\text{int}})}. 
 \end{aligned}
 \end{equation} 
By \eqref{eq_9_3} and \eqref{eq_9_4}, we have
\begin{equation}
 \label{eq_A_8}
\begin{aligned}
 |\langle q^{(1)}(v_0+v_1), \psi \rangle_{M^{\text{int}}}| & \le \|q^{(1)}\|_{L^\infty(M^0)} \|v_0 +v_1\|_{L^2(M)} \|\psi\|_{L^2(M)} \\
 & \le \mathcal{O}(\lambda^{\frac{n-1}{4}+\frac{1}{2}}) \|\psi\|_{H^1(M^{\text{int}})}. 
\end{aligned}
\end{equation}
The estimate \eqref{eq_A_r1} follows from \eqref{eq_A_6}, \eqref{eq_A_7} and \eqref{eq_A_8}. 

Let us show that there exists a solution $u_2 \in H^3(M^{\text{int}})$ of $L_{-\overline{X^{(2)}},-\div(\overline{X^{(2)}})+ \overline{q^{(2)}}} u_2 = 0$ in $M$ of the form
\begin{equation}
\label{eq_A_9}
u_2 = v_0 + v_1 + r_2
\end{equation}
where 
$r_2 \in H^3(M^{\text{int}})$ with 
\begin{equation}
 \label{eq_A_r2}
 \|r_2\|_{H^3(M^{\text{int}})} \le \mathcal{O} (\lambda^{\frac{n-1}{4}+\frac{1}{2}}).
 \end{equation} 
Applying Proposition \ref{prop_solvability} with $h>0$ small but fixed to the equation, 
\begin{equation}
 \label{eq_A_5_2}
 L_{-\overline{X^{(2)}},-\div(\overline{X^{(2)}})+ \overline{q^{(2)}}} r_2=(\overline{X^{(2)}}+\div(\overline{X^{(2)}})-\overline{q^{(2)}})(v_0+v_1)\quad  \text{in}\quad M,
\end{equation}
  we conclude the existence of $r_2 \in H^1(M^{\text{int}})$ such that 
 \begin{equation}
 \label{eq_A_6_2}
 \|r_2\|_{H^3(M^{\text{int}})} \le \mathcal{O}(1) \|(\overline{X^{(2)}}+\div(\overline{X^{(2)}})-\overline{q^{(2)}})(v_0+v_1) \|_{H^{-1}(M^{\text{int}})}. 
 \end{equation}
To bound the norm in the right hand side of \eqref{eq_A_6_2},  we let $\psi \in C^\infty_0(M^{\text{int}})$, and using \eqref{eq_Hardy},   \eqref{eq_100_1_1}, \eqref{eq_9_3}, \eqref{eq_9_4}, \eqref{eq_Brown_Salo_2_20},
we get
\begin{equation}
 \label{eq_A_6_2_0}
\begin{aligned}
|\langle &\div(\overline{X^{(2)}})(v_0+v_1), \psi \rangle_{M^{\text{int}}}|=\bigg|\int \overline{X^{(2)}}((v_0+v_1)\psi)dV_g \bigg|\\
&\le \bigg| \int \psi \overline{X^{(2)}}(v_0+v_1)dV_g \bigg|+ \bigg|\int (v_0+v_1)\overline{X^{(2)}}(\psi) dV_g \bigg|\\
&\le \mathcal{O}(1)\|\delta d (v_0+v_1)\|_{L^2(M)}\|\psi\|_{H^1(M^{\text{int}})}+\mathcal{O}(1)\|v_0+v_1\|_{L^2(M)}\|\psi\|_{H^1(M^{\text{int}})}\\
&\le \mathcal{O}(\lambda^{\frac{n-1}{4}+\frac{1}{2}})\|\psi\|_{H^1(M^{\text{int}})}.
\end{aligned}
\end{equation}
The bound \eqref{eq_A_r2} follows from \eqref{eq_A_6_2}, \eqref{eq_A_6_2_0}, \eqref{eq_A_7}, \eqref{eq_A_8}.

The next step is to substitute the solution $u_1$ and $u_2$, given in \eqref{eq_A_4} and \eqref{eq_A_9} into the integral identity \eqref{eq_integral_identity}, multiply by $\lambda^{-\frac{(n-1)}{2}}$ and compute the limit as $\lambda \to 0$. 
In doing so, we write
\begin{equation}
\label{eq_A_11}
I :=\lambda^{-\frac{(n-1)}{2}} \int_M X(u_1)\overline{u_2} +qu_1\overline{u_2} dV_g = I_1 + I_2 +I_3 + I_4 + I_5 + I_6,
\end{equation} 
where 
\begin{align*}
& I_1 = \lambda^{-\frac{(n-1)}{2}} \int_M X(v_0) \overline{v_0} dV_g,\quad I_2 = \lambda^{-\frac{(n-1)}{2}} \int_M X(v_0) \overline{v_1} dV_g,\\
&  I_3 = \lambda^{-\frac{(n-1)}{2}} \int_M X(v_0) \overline{r_2} dV_g, \quad I_4 = \lambda^{-\frac{(n-1)}{2}} \int_M X(v_1) \overline{u_2} dV_g,\\
& I_5 = \lambda^{-\frac{(n-1)}{2}} \int_M X(r_1) \overline{u_2} dV_g,\quad I_6 = \lambda^{-\frac{(n-1)}{2}} \int_M q u_1 \overline{u_2} dV_g.
\end{align*}

Let us compute $\lim_{\lambda \to 0} I_1$. To that end, writing $X=X_j\p_{x_j}$, we have 
\begin{equation}
\label{eq_A_12}
Xv_0 = e^{\frac{i}{\lambda}(\tau' \cdot x' + i x_n) } [ \lambda^{-\frac{1}{2}} (X\eta) (\frac{x}{\lambda^{\frac{1}{2}}}) + i \lambda^{-1} X(x)\cdot (\tau', i) \eta (\frac{x}{\lambda^{\frac{1}{2}}})],
\end{equation}
and 
\begin{equation}
\label{eq_A_13}
Xv_0 \overline{v_0} = e^{-\frac{2x_n}{\lambda}}[  \lambda^{-\frac{1}{2}} (X\eta) (\frac{x}{\lambda^{\frac{1}{2}}}) \eta (\frac{x}{\lambda^{\frac{1}{2}}}) + i \lambda^{-1} X(x)\cdot (\tau', i) \eta^2 (\frac{x}{\lambda^{\frac{1}{2}}}) ]. 
\end{equation}
Making the change of variable $y'= \frac{x'}{\lambda^{1/2}}$, $y_n= \frac{x_n}{\lambda}$, using that $X \in C(M, TM)$, $ \eta$ has compact support, \eqref{eq_Laplace_boundary-nc_2} and \eqref{eq_int_eta_1}, we get 
\begin{equation}
\label{eq_A_14}
\begin{aligned}
\lim_{\lambda \to 0} &I_1 = \lim_{\lambda \to 0} \int_{\R^{n-1}} \int^\infty_{0} e^{-2y_n} \lambda^{\frac{1}{2}} (X\eta) (y',\lambda^{\frac{1}{2}}y_n) \eta(y',\lambda^{\frac{1}{2}}y_n) |g(\lambda^{\frac{1}{2}}y',\lambda y_n)|^{\frac{1}{2}} dy_n dy' \\
& + \lim_{\lambda \to 0} \int_{\R^{n-1}} \int^\infty_{0} e^{-2y_n}  iX(\lambda^{\frac{1}{2}}y',\lambda y_n) \cdot (\tau', i ) \eta^2 (y',\lambda^{\frac{1}{2}}y_n)  |g(\lambda^{\frac{1}{2}}y',\lambda y_n)|^{\frac{1}{2}} dy_n dy'\\ 
& = \frac{i}{2}X(0)\cdot (\tau', i). 
\end{aligned}
\end{equation}

The fact that $v_1\in H^1_0(M^{\text{int}})$ together with the estimates \eqref{eq_Hardy},  \eqref{eq_Brown_Salo_2_20}, \eqref{eq_9_15} gives that 
\begin{equation}
\label{eq_A_15}
 |I_2| \le \mathcal{O}( \lambda^{-\frac{(n-1)}{2}}) \|X\|_{L^\infty(M)} \| \delta dv_0\|_{L^2(M)}   \|\frac{v_1}{\delta}\|_{L^2(M)} = \mathcal{O}(\lambda^{\frac{1}{2}}). 
\end{equation}

To estimate $I_3$, first assume that $(M,g)$ is embedded in a compact smooth manifold $(N,g)$ without boundary of the same dimension. Let us extend $X\in C(M,TM)$ to a continuous vector field on $N$, and still write $X\in C(N,TN)$.   
Using a partition of unity argument together with a regularization in each coordinate patch, we see that there exists a family $X_\tau\in C^\infty(N,TN)$ such that 
\begin{equation}
\label{eq_A_15_reg}
\|X-X_\tau\|_{L^\infty}=o(1),\quad \|X_\tau\|_{L^\infty}=\mathcal{O}(1), \quad \|\nabla X_\tau\|_{L^\infty}=\mathcal{O}(\tau^{-1}),\quad \tau\to 0.
\end{equation}
We write 
\begin{equation}
\label{eq_A_15_reg_1}
I_3=I_{3,1}+I_{3,2},
\end{equation}
where
\begin{equation}
\label{eq_A_15_reg_2}
I_{3,1} = \lambda^{-\frac{(n-1)}{2}} \int_M (X-X_\tau)(v_0) \overline{r_2} dV_g,\quad I_{3,2} = \lambda^{-\frac{(n-1)}{2}} \int_M X_\tau(v_0) \overline{r_2} dV_g.
\end{equation}
Using \eqref{eq_A_15_reg}, \eqref{eq_Brown_Salo_2_18_new_mnfld}, \eqref{eq_A_r2}, we get 
\begin{equation}
\label{eq_A_15_reg_3}
|I_{3,1}|\le \mathcal{O}(\lambda^{-\frac{(n-1)}{2}})\|X-X_\tau\|_{L^\infty(M)}\|dv_0\|_{L^2(M)}\|r_2\|_{L^2(M)}=o(1),
\end{equation}
as $\tau\to 0$. To estimate $I_{3,2}$, integrating by parts, we obtain that 
\begin{equation}
\label{eq_A_15_reg_4}
I_{3,2} = J_{1}+J_{2}+J_3, 
\end{equation}
where 
\begin{equation}
\label{eq_A_15_reg_5}
\begin{aligned}
J_{1}=-\lambda^{-\frac{(n-1)}{2}} \int_M v_0 &X_\tau(\overline{r_2}) dV_g,\quad J_{2}=-\lambda^{-\frac{(n-1)}{2}} \int_M  \div(X_\tau)v_0\overline{r_2}dV_g, \\
&J_{3}=\lambda^{-\frac{(n-1)}{2}}\int_{\p M} (\nu\cdot X_\tau)v_0\overline{r_2}dS_g.
\end{aligned}
\end{equation}
Using \eqref{eq_A_15_reg}, \eqref{eq_A_r2}, \eqref{eq_9_3}, we get 
\begin{equation}
\label{eq_A_15_reg_6}
\begin{aligned}
&|J_{1}|\le \mathcal{O}(\lambda^{-\frac{(n-1)}{2}})\|X_\tau\|_{L^\infty(M)}\|v_0\|_{L^2(M)}\|dr_2\|_{L^2(M)}=\mathcal{O}(\lambda),\\
&|J_2|\le  \mathcal{O}(\lambda^{-\frac{(n-1)}{2}})\|\div X_\tau\|_{L^\infty(M)}\|v_0\|_{L^2(M)}\|r_2\|_{L^2(M)}=\mathcal{O}(\tau^{-1}\lambda).
\end{aligned}
\end{equation}
Using \eqref{eq_Brown_Salo_2_20_boundary}, \eqref{eq_A_15_reg}, \eqref{eq_A_r2}, and the trace theorem, we obtain that 
\begin{equation}
\label{eq_A_15_reg_7}
|J_3|\le  \mathcal{O}(\lambda^{-\frac{(n-1)}{2}})\|\nu\cdot X_\tau\|_{L^\infty(M)}\|v_0\|_{L^2(\p M)}\|r_2\|_{H^1(M)}=\mathcal{O}(\lambda^{1/2}).
\end{equation}
Choosing $\tau=\lambda^{1/2}$, we conclude from \eqref{eq_A_15_reg_1}, \eqref{eq_A_15_reg_2}, \eqref{eq_A_15_reg_3}, \eqref{eq_A_15_reg_4}, \eqref{eq_A_15_reg_5}, \eqref{eq_A_15_reg_6}, \eqref{eq_A_15_reg_7} that 
\begin{equation}
\label{eq_A_16}
 |I_3| =o(1), \quad \lambda\to 0. 
\end{equation}

Now \eqref{eq_9_3}, \eqref{eq_9_4}, \eqref{eq_A_r2} imply that 
\begin{equation}
\label{eq_A_17_0}
\|u_2\|_{L^2}=\mathcal{O}(\lambda^{\frac{n-1}{4}+\frac{1}{2}}). 
\end{equation}
Using \eqref{eq_A_17_0} together with  \eqref{eq_9_15}, we have 
\begin{equation}
\label{eq_A_17}
 |I_4| \le  \mathcal{O}(\lambda^{-\frac{(n-1)}{2}})  \|dv_1\|_{L^2(M)} \|u_2 \|_{L^2(M)}  = \mathcal{O} (\lambda^{\frac{1}{2}}). 
\end{equation}
 
Using \eqref{eq_A_17_0} together with \eqref{eq_A_r1}, we get
\begin{equation}
\label{eq_A_18}
 |I_5| \le \mathcal{ O}(\lambda^{-\frac{(n-1)}{2}})   \|dr_1\|_{L^2(M)} \|u_2 \|_{L^2(M)}  = \mathcal{O} (\lambda). 
\end{equation}

Last let us estimate $| I_6 |$. Using \eqref{eq_A_17_0} and similar bound for $u_1$, we see that 
\begin{equation}
\label{eq_A_19}
 |I_6| \le  \mathcal{O}(\lambda^{-\frac{(n-1)}{2}} )  \| q \|_{L^\infty(M)}\|u_1 \|_{L^2(M)}\|u_2 \|_{L^2(M)}= \mathcal{O}(\lambda). 
\end{equation}

Now it follows from \eqref{eq_A_11}, \eqref{eq_A_14}, \eqref{eq_A_15}, \eqref{eq_A_16}, \eqref{eq_A_17}, \eqref{eq_A_18} and \eqref{eq_A_19} that 
\[
\lim_{\lambda \to 0} I = \frac{i}{2}X(0)\cdot (\tau', i)=0,
\]
and therefore, 
\[
X^{(1)}(0)\cdot (\tau', i) = X^{(2)}(0)\cdot (\tau', i),
\]
for all $\tau' \in \R^{n-1}$. This completes the proof of Proposition \ref{boundary determination_X}.
\end{proof}
\end{appendix}

\section*{Acknowledgements}
The author would like to thank Katya Krupchyk for her support and guidance. The research is partially supported by the National Science Foundation (DMS 1815922).

\end{document}